\documentclass[pdflatex,sn-mathphys-num]{sn-jnl}

\usepackage[table]{xcolor}

\usepackage{comment}
\usepackage{graphicx}%
\usepackage{amsmath,amssymb,amsfonts}%
\usepackage{amsthm}%
\usepackage{textcomp}%
\usepackage{mathtools}
\usepackage[T1]{fontenc}

\theoremstyle{thmstyleone}%
\newtheorem{theorem}{Theorem}[section]

\newtheorem{proposition}[theorem]{Proposition}
\newtheorem{lemma}[theorem]{Lemma}
\newtheorem{corollary}[theorem]{Corollary}

\theoremstyle{thmstyletwo}%
\newtheorem{example}[theorem]{Example}
\newtheorem{remark}[theorem]{Remark}

\theoremstyle{thmstylethree}%
\newtheorem{definition}[theorem]{Definition}

\newcommand{\PosId}{\operatorname{Id}^{+}}

\newcommand{\tprod}[2]{#1\ast_{\tau}#2}

\newcommand{\layer}[1]{L_{#1}}
\newcommand{\Ctau}{\mathcal C_{\tau}}
\newcommand{\Cm}{\mathcal C_{\mathrm m}}
\newcommand{\Ltau}{\mathcal L_{\tau}}
\newcommand{\Lm}{\mathcal L_{\mathrm m}}

\makeatletter
\newcommand{\taggeditem}[2]{%
 \item[#1]%
 \def\@currentlabel{#1}%
 \label{#2}%
}
\makeatother

\begin{document}

\title[Categorical rigid direct-system representations]{Categorical forms of the rigid direct-system representation for finite local-unit-aligned totally ordered monoids}

\author*[1,2]{\fnm{S\'andor} \sur{Jenei}}\email{jenei.sandor@uni-eszterhazy.hu, jenei@ttk.pte.hu}

\affil*[1]{\orgdiv{Institute of Mathematics and Informatics}, \orgname{Eszterh\'azy K\'aroly Catholic University}, \orgaddress{\country{Hungary}}}
\affil[2]{\orgdiv{Institute of Mathematics and Informatics}, \orgname{University of P\'ecs}, \orgaddress{\country{Hungary}}}

\abstract{We study the functorial and categorical structure of the canonical rigid direct-system representation of finite local-unit-aligned totally ordered monoids.
The local-unit map \(\tau\) induces a canonical \(\tau\)-multiplication-coherent decomposition into component monoids, and the associated representation reconstructs the original ordered monoid from a finite chain-indexed rigid direct system whose proper transition maps are unit-constant.
The present paper identifies the morphism classes for which this representation is categorical.
First, we prove an equivalence between finite local-unit-aligned totally ordered monoids with strict block morphisms and rigid direct systems with directed-order-compatible system morphisms.
Second, we prove an intrinsic equivalence for \(\tau\)-compatible homomorphisms, that is, isotone unital homomorphisms commuting with the local-unit map.
In this second setting, distinct positive idempotents and hence distinct canonical components may collapse to a single target component; on the direct-system side this is represented by non-injective isotone index maps together with component maps satisfying the corresponding collapse and absorption axioms.
Thus the canonical rigid direct-system representation is functorial both for strict block morphisms and for intrinsic \(\tau\)-compatible homomorphisms.}



\maketitle

\section{Introduction}\label{IntroONE}

Finite totally ordered monoids provide a setting in which algebraic and order-theoretic structure are tightly constrained.
In the local-unit-aligned case, each element has coinciding greatest right and left local units.
The resulting local-unit map \(\tau\colon M\to \PosId(\mathbf M)\) induces a canonical stratification of the positive idempotent skeleton, and the \(\tau\)-multiplication-coherent part of this stratification supports a rigid direct-system representation.
The purpose of this paper is to study the categorical content of that representation.

The object-level representation, recalled in Section~\ref{sec:recap}, assigns to a finite local-unit-aligned totally ordered monoid a finite chain-indexed direct system of component monoids.
These components are induced intrinsically by the canonical \(\tau\)-multiplication-coherent partition, and the ambient monoid is recovered from them by the directed lexicographic order and the corresponding direct-system multiplication.
A finite rigidity phenomenon is essential: proper transition maps in the canonical system are forced to be constant at the unit of the target component.
Thus the representation is not an arbitrary direct-system decomposition, but an ordinal-sum-like rigid assembly in which mixed products are governed by absorption toward the upper component.

The guiding question of the present paper is functorial.
Once a monoid has been replaced by its canonical rigid direct system, which homomorphisms of monoids are represented exactly by morphisms of the corresponding systems?
The answer is not unique, because two natural morphism classes occur.

The first class consists of what we call \emph{strict block morphisms}.
Such a morphism is required to send each canonical component into a canonical component and to send the distinguished block unit to the corresponding target block unit.
Here ``strict'' refers to block-unit preservation, not to injectivity on the block chain: several source blocks may still be sent to the same target block, but the direct-system compatibility then forces the corresponding lower component maps to be unit-constant.
This produces the cleanest categorical expression of the reconstruction theorem: finite local-unit-aligned totally ordered monoids with such morphisms are equivalent to rigid direct systems with directed-order-compatible componentwise morphisms.

The second class is intrinsic.
A homomorphism should preserve the local-unit geometry itself, meaning that it commutes with \(\tau\).
Such a \(\tau\)-compatible homomorphism need not be injective on positive idempotents.
Consequently, it may identify distinct positive idempotents and collapse several canonical components of the source into one canonical component of the target.
On the direct-system side this is represented by an isotone, not necessarily injective, map between index chains together with component morphisms satisfying collapse and absorption conditions.
Rigidity is crucial here: the unit-constancy of proper transition maps is exactly what makes such component collapse compatible with multiplication and order.

\medskip
\noindent\textbf{Main theorem, informally.}
The canonical rigid direct-system representation yields two quasi-inverse categorical equivalences.
The first is an equivalence between the strict block category of finite local-unit-aligned totally ordered monoids and the category of rigid direct systems with directed-order-compatible system morphisms.
The second is an equivalence between the category of finite local-unit-aligned totally ordered monoids with \(\tau\)-compatible homomorphisms and the category of rigid direct systems with \(\tau\)-compatible collapse morphisms.
\medskip

This categorical viewpoint shows that the canonical decomposition is not merely an object-level reconstruction device.
It is a functorial invariant whose morphism theory has two complementary forms.
The strict block form remembers canonical components and their distinguished block units component by component, while still allowing several source components to be mapped into one target component when the strict block and system-compatibility conditions are satisfied.
The intrinsic form is controlled instead by the local-unit equation \(\tau_{\mathbf N}f=f\tau_{\mathbf M}\); it also permits collapse of components, but records that collapse through the local-unit geometry rather than through explicit preservation of block units.

The distinction between these two morphism theories is one of the main points of the paper.
In the strict block setting, the index map records how canonical components are transported.
In the \(\tau\)-compatible setting, the index map records how local-unit strata may be identified.
The second situation is subtler: a source component may not simply be sent into an arbitrary target component; the collapse must be compatible with the target's canonical \(\tau\)-multiplication-coherent partition.
The proofs show that \(\tau\)-multiplication-cohesiveness of the rigid components and unit-constancy of the transition maps are exactly the structural facts needed to make this collapse functorial.

The paper also clarifies the role of the representation in the classical component-and-transition-map tradition.
Clifford's strong-semilattice representation and P\l{}onka's sums are precedents for recovering a global algebraic object from local components and connecting maps \cite{Clifford1941,CliffordPreston1961,Plonka1967,Plonka1968}.
In categorical language, such decompositions may be viewed as indexed families of algebraic objects, or functors from an ordered index category, together with a construction which assembles the indexed data into a total object.
This viewpoint appears explicitly in Bowman's construction functors for topological semigroups \cite{Bowman1975}, in the categorical treatment of semilattice sums and semilattice representations \cite{RomanowskaSmith1991,RomanowskaSmith1997}, and in Zawadowski's monadic treatment of generalized P\l{}onka sums and products via lax and oplax morphisms of monads \cite{Zawadowski2015}.
Related recent work on duality and enriched forms of P\l{}onka sums shows that this categorical direct-system viewpoint remains active beyond the original semilattice-sum setting \cite{Bonzio2018,FuscoPaoli2025}.

There is also a categorical side to semigroup decompositions themselves.
The greatest semilattice image of a semigroup is a reflection from the category of semigroups to the category of semilattices, and its limit-preservation properties have been studied categorically \cite{JanelidzeLaanMarki2008}.
A different, but important, categorical structure theorem is the Ehresmann--Schein--Nambooripad correspondence between inverse semigroups and inductive groupoids, together with its later extensions \cite{Hollings2012,DeWolfPronk2018}.
More generally, Grothendieck-construction methods have also been used to classify monoid extensions \cite{Manuell2022}.
These results provide categorical context for the present paper, but the present equivalences are of a different kind: they concern the functoriality of a canonical local-unit decomposition rather than a passage to groupoids, fibrations, or general semilattice quotients.

In ordered semigroup theory, ordinal-sum phenomena go back to Clifford's work on naturally totally ordered commutative semigroups \cite{Clifford1954,Clifford1958}.
The present setting is different: no inverse operation is assumed, commutativity is not assumed, and the identity need not lie at either endpoint of the chain.
The components and transition maps are induced by the local-unit map \(\tau\), and finiteness plus total order force the rigid unit-constant behavior that makes the categorical comparison possible.
Thus the categorical contribution here is not a general categorical reformulation of P\l{}onka sums or of Clifford semigroups, but a pair of equivalence theorems for the morphism theory naturally attached to the canonical rigid local-unit representation.

The local-unit stratification used here is also compatible with the residuated motivation from which such layers originally arose \cite{Jenei2022GroupRepr}.
Nevertheless, the results of this paper are monoid-theoretic and categorical rather than residuation-theoretic: the arguments use only the ordered-monoid structure, the local-unit map, the canonical rigid direct system, and the corresponding morphism classes.
The final remark records that residuated homomorphisms naturally enter the \(\tau\)-compatible category, because they automatically commute with the local-unit map in the relevant setting.

The paper is organized as follows.
Section~\ref{sec:recap} recalls the local-unit terminology, the canonical \(\tau\)-multiplication-coherent decomposition, the rigid direct-system reconstruction, and the structural facts needed later.
Section~\ref{sec:categorical-equivalence} proves the categorical equivalence for strict block morphisms.
Section~\ref{sec:taucomp-categorical-equivalence} proves the intrinsic categorical equivalence for \(\tau\)-compatible homomorphisms and \(\tau\)-compatible collapse morphisms of rigid direct systems.

\section{Preliminaries and the rigid representation theorem}\label{sec:recap}

This section fixes the local-unit terminology and restates the algebraic representation facts used as input for the categorical arguments.
The object-level representation theorem is proved in \cite{JeneiLUARepresentation}; here it is recalled in the form needed for functoriality: canonical \(\tau\)-multiplication-coherent components, their block units, the rigid direct-system reconstruction, and the recovery of canonical components from a rigid system.

A \emph{finite totally ordered monoid} is a structure \(\mathbf M=\langle M,\le,\cdot,e\rangle\) in which \(\langle M,\le\rangle\) is a finite chain, \(\langle M,\cdot,e\rangle\) is a monoid, and multiplication is isotone in both arguments.
For \(x\in M\), put
\[
R(x):=\{z\in M:xz\le x\},
\qquad
L(x):=\{z\in M:zx\le x\}.
\]
The monoid \(\mathbf M\) is \emph{local-unit-aligned} if the greatest elements of \(R(x)\) and \(L(x)\) coincide for every \(x\); their common value is denoted \(\tau_{\mathbf M}(x)\), or simply \(\tau(x)\).
Thus \(x\tau(x)=x=\tau(x)x\), and \(\tau(x)\) is a positive idempotent.
We write
\[
\PosId(\mathbf M)=\{u\in M:u^2=u\text{ and }e\le u\}.
\]
On \(\PosId(\mathbf M)\), multiplication is the maximum operation.

For \(U\subseteq\PosId(\mathbf M)\), the \(U\)-layer is
\[
\layer{U}:=\tau^{-1}(U)=\{x\in M:\tau(x)\in U\}.
\]
For \(u\in\PosId(\mathbf M)\), we write \(\layer{u}\) for the singleton layer \(\layer{\{u\}}\), that is, \(\layer{u}:=\{x\in M:\tau(x)=u\}\).
For \(A,B\subseteq\PosId(\mathbf M)\), define the \emph{\(\tau\)-saturated product} by
\[
\tprod{A}{B}:=\{\tau(xy):x\in\layer{A},\ y\in\layer{B}\}.
\]
A nonempty subset \(A\subseteq\PosId(\mathbf M)\) is \emph{\(\tau\)-stable} if \(\tprod{A}{A}=A\), and a partition of \(\PosId(\mathbf M)\) is \(\tau\)-stable if all its blocks are \(\tau\)-stable.
The \emph{\(\tau\)-cohesive partition} \(\Ctau(\mathbf M)\) is the finest \(\tau\)-stable partition of \(\PosId(\mathbf M)\), and
\[
\Ltau(\mathbf M):=\{\layer{A}:A\in\Ctau(\mathbf M)\}
\]
is the corresponding \emph{\(\tau\)-cohesive decomposition} of \(\mathbf M\).
The monoid \(\mathbf M\) is \emph{\(\tau\)-cohesive} if this partition is trivial, equivalently if \(\Ctau(\mathbf M)=\{\PosId(\mathbf M)\}\).

A partition \(\mathcal P\) of \(\PosId(\mathbf M)\) is \emph{\(\tau\)-multiplication-coherent} if, for all blocks \(A,B\in\mathcal P\), the set \(\tprod{A}{B}\) is contained in a unique block of \(\mathcal P\).
The \emph{\(\tau\)-multiplication-coherent partition} \(\Cm(\mathbf M)\) is the finest such partition, and
\[
\Lm(\mathbf M):=\{\layer{A}:A\in\Cm(\mathbf M)\}
\]
is the corresponding \emph{\(\tau\)-multiplication-coherent decomposition} of \(\mathbf M\).
The monoid \(\mathbf M\) is \emph{\(\tau\)-multiplication-cohesive} if \(\Cm(\mathbf M)=\{\PosId(\mathbf M)\}\).
The partition \(\Ctau(\mathbf M)\) always refines \(\Cm(\mathbf M)\).

If \(A\in\Cm(\mathbf M)\), write \(M_A:=\layer{A}\).
The associated component monoid is
\[
\mathbf M_A=\langle M_A,\le|_{M_A},\cdot|_{M_A\times M_A},e_A\rangle,
\qquad
e_A:=\min A.
\]
We call \(e_A\) the block unit of \(A\).

\begin{theorem}\label{thm:component-local-unit-aligned-monoid}
For every \(A\in\Cm(\mathbf M)\), the component \(\mathbf M_A\) is a finite local-unit-aligned totally ordered monoid, and its local-unit map is the restriction of the ambient one: \(\tau_{\mathbf M_A}(x)=\tau_{\mathbf M}(x)\) for all \(x\in \layer{A}\).
\end{theorem}

\begin{lemma}\label{lem:intrinsic-tau-on-component}
Let \(A\in\Cm(\mathbf M)\).
Then \(\PosId(\mathbf M_A)=A\).
Moreover, for every \(U\subseteq A\), the \(U\)-layer computed inside \(\mathbf M_A\) equals the ambient \(U\)-layer restricted to \(M_A\), and intrinsic and ambient \(\tau\)-saturated products agree for subsets of \(A\).
\end{lemma}

For \(A,B\in\Cm(\mathbf M)\), product coherence determines a block \(A\vee B\): it is the unique block containing the \(\tau\)-values of products of elements from the corresponding layers.

\begin{lemma}\label{lem:block-join-min}
For \(A,B\in\Cm(\mathbf M)\), the block \(A\vee B\) is the unique block containing \(\max\{e_A,e_B\}\).
Consequently, the relation \(A\le_{\min}B\) defined by \(e_A\le e_B\) is a total order on \(\Cm(\mathbf M)\), and \(A\vee B=\max_{\le_{\min}}\{A,B\}\).
\end{lemma}

For \(A\le_{\min}B\), the canonical transition map is \(\rho_{A,B}\colon M_A\to M_B\), \(\rho_{A,B}(x):=xe_B=e_Bx\).
These maps form the canonical direct system \(\mathfrak D(\mathbf M)=((\mathbf M_A)_{A\in\Cm(\mathbf M)},(\rho_{A,B})_{A\le_{\min}B})\).

\begin{theorem}[Canonical direct-system recovery]\label{thm:canonical-direct-system}
Let \(\mathbf M\) be a finite local-unit-aligned totally ordered monoid.
Then \(\mathfrak D(\mathbf M)\) is a finite chain-indexed direct system and recovers the ambient order and multiplication.
More explicitly:
\begin{enumerate}
\item\label{item:direct-disjoint-union} \(M=\bigsqcup_{A\in\Cm(\mathbf M)}M_A\);

\item\label{item:direct-product-recovery} for \(x\in M_A\), \(y\in M_B\), and \(C=A\vee B\), the product in \(\mathbf M\) is
\[
xy=\rho_{A,C}(x)\rho_{B,C}(y),
\]
computed in the component monoid \(\mathbf M_C\);

\item\label{item:direct-order-recovery} if \(A\le_{\min}B\), then \(x\le y\) in \(\mathbf M\) iff \(\rho_{A,B}(x)\le_B y\), while if \(B<_{\min}A\), then \(x\le y\) iff \(x<_A\rho_{B,A}(y)\).
\end{enumerate}
\end{theorem}

A \emph{finite chain-indexed direct system} of finite local-unit-aligned totally ordered monoids is a family
\[
\mathbf D=\bigl((\mathbf M_i)_{i\in I},(\rho_{i,j})_{i\le j}\bigr),
\qquad
\mathbf M_i=\langle M_i,\le_i,\cdot_i,e_i\rangle
\]
indexed by a nonempty finite chain \((I,\le)\), where each \(\mathbf M_i\) is a finite local-unit-aligned totally ordered monoid, each \(\rho_{i,j}\colon\mathbf M_i\to\mathbf M_j\) is an isotone unital monoid homomorphism, \(\rho_{i,i}=\mathrm{id}_{M_i}\), and \(\rho_{j,k}\circ\rho_{i,j}=\rho_{i,k}\) whenever \(i\le j\le k\).
For an abstract direct system, the external union \(\bigsqcup_{i\in I}M_i\) is understood in the tagged-disjoint-copy sense.

The \emph{directed lexicographic order} on \(\bigsqcup_{i\in I}M_i\) is the reflexive closure of the strict order defined as follows: for \(x\in M_i\) and \(y\in M_j\), putting \(m:=\max\{i,j\}\),
\[
x<y
\quad\Longleftrightarrow\quad
\rho_{i,m}(x)<_m\rho_{j,m}(y)
\ \text{or}\
\bigl(\rho_{i,m}(x)=\rho_{j,m}(y)\text{ and }i<j\bigr).
\]
A direct system is \emph{unit-constant} if every proper transition map \(\rho_{i,j}\), \(i<j\), sends every element of \(M_i\) to the identity \(e_j\) of the target component.
In the finite local-unit-aligned totally ordered setting, the canonical system is unit-constant.
In the categorical sections below, a rigid direct system is the unit-constant case with \(\tau\)-multiplication-cohesive constituents.

Conversely, if
\[
\mathbf D=\bigl((\mathbf M_i)_{i\in I},(\rho_{i,j})_{i\le j}\bigr)
\]
is a unit-constant finite chain-indexed direct system of finite local-unit-aligned totally ordered monoids, then the reconstructed monoid \(\mathfrak M(\mathbf D)\) has universe \(\bigsqcup_{i\in I}M_i\), directed lexicographic order, global identity \(e_{\min I}\), and multiplication
\[
xy:=\rho_{i,k}(x)\cdot_k\rho_{j,k}(y)
\qquad
(x\in M_i,\ y\in M_j,\ k:=\max\{i,j\}).
\]
Each constituent \(M_i\) is an ordered subsemigroup of \(\mathfrak M(\mathbf D)\), with its original multiplication and order, and the ambient local-unit map restricts to the intrinsic local-unit map of \(\mathbf M_i\).
No \(\tau\)-multiplication-cohesiveness assumption on the constituents is needed for this reconstruction; that hypothesis is only needed when one wants the construction-level decomposition to coincide with the canonical \(\tau\)-multiplication-coherent decomposition of the resulting monoid.

\begin{remark}[Working forms of the directed lexicographic order]\label{rem:directed-lex-order-working-forms}
In \(\mathfrak M(\mathbf D)\), if \(x\in M_i\) and \(y\in M_j\), then for \(i\le j\) one has \(x\le y\) iff \(\rho_{i,j}(x)\le_j y\), while for \(j<i\) one has \(x\le y\) iff \(x<_i\rho_{j,i}(y)\).
Equivalently, for strict order, \(i<j\) gives \(x<y\) iff \(\rho_{i,j}(x)\le_j y\), and \(j<i\) gives \(x<y\) iff \(x<_i\rho_{j,i}(y)\).
\end{remark}

\begin{corollary}[Rigidity of canonical transition maps]\label{cor:canonical-transition-maps-unit-constant}
If \(A<_{\min}B\) in \(\Cm(\mathbf M)\), then \(\rho_{A,B}(x)=e_B\) for all \(x\in M_A\).
\end{corollary}

\begin{corollary}[The canonical components are \(\tau\)-multiplication-cohesive]\label{cor:canonical-components-already-terminal-direct}
For every \(A\in\Cm(\mathbf M)\), the component monoid \(\mathbf M_A\) is \(\tau\)-multiplication-cohesive.
\end{corollary}

\begin{theorem}[Canonical rigid direct-system representation theorem]\label{thm:representation}
The following hold.
\begin{enumerate}
\taggeditem{\textup{(i)}}{item:representation-canonical-direct-system}
Every finite local-unit-aligned totally ordered monoid \(\mathbf M\) admits the canonical unit-constant direct system
\[
\mathfrak D(\mathbf M)
=
\bigl((\mathbf M_A)_{A\in\Cm(\mathbf M)},(\rho_{A,B})_{A\le_{\min}B}\bigr)
\]
of \(\tau\)-multiplication-cohesive finite local-unit-aligned totally ordered monoids over the finite chain \((\Cm(\mathbf M),\le_{\min})\).

\taggeditem{\textup{(ii)}}{item:representation-reconstruction-from-direct-system}
Conversely, every unit-constant finite chain-indexed direct system
\[
\mathbf D=\bigl((\mathbf M_i)_{i\in I},(\rho_{i,j})_{i\le j}\bigr)
\]
of finite local-unit-aligned totally ordered monoids, with no cohesiveness assumption on the constituents, reconstructs a finite local-unit-aligned totally ordered monoid \(\mathfrak M(\mathbf D)\).
Its universe is the disjoint union of the constituents, its order is the directed lexicographic order, and its multiplication is
\[
xy:=\rho_{i,k}(x)\cdot_k\rho_{j,k}(y)
\qquad
(x\in M_i,\ y\in M_j,\ k:=\max\{i,j\}).
\]

\taggeditem{\textup{(iii)}}{item:representation-canonical-recovery}
For every finite local-unit-aligned totally ordered monoid \(\mathbf M\), the canonical comparison \(\mathfrak M(\mathfrak D(\mathbf M))\cong\mathbf M\) is an isomorphism of ordered monoids.

\taggeditem{\textup{(iv)}}{item:representation-direct-system-recovery}
If \(\mathbf D=\bigl((\mathbf M_i)_{i\in I},(\rho_{i,j})_{i\le j}\bigr)\) is a unit-constant finite chain-indexed direct system of finite local-unit-aligned totally ordered monoids, then the canonical decomposition \(\Lm(\mathfrak M(\mathbf D))\) refines the construction-level partition \(\{M_i:i\in I\}\).
If, moreover, every constituent monoid \(\mathbf M_i\) is \(\tau\)-multiplication-cohesive, then the construction-level layer decomposition is exactly the canonical \(\tau\)-multiplication-coherent layer decomposition.
More precisely, putting
\[
P_i:=\PosId(\mathfrak M(\mathbf D))\cap M_i,
\]
one has
\[
\Cm(\mathfrak M(\mathbf D))=\{P_i:i\in I\},
\qquad
\Lm(\mathfrak M(\mathbf D))=\{M_i:i\in I\}.
\]
In this case \(i\mapsto P_i\) is an order isomorphism from \(I\) onto \(\Cm(\mathfrak M(\mathbf D))\), the canonical component monoid indexed by \(P_i\) is exactly \(\mathbf M_i\), its block unit is \(e_i\), and, for \(i\le j\), the recovered canonical transition map
\[
\widehat\rho_{P_i,P_j}\colon M_i\to M_j
\]
coincides with the original transition map \(\rho_{i,j}\).
Thus the recovered canonical direct system is canonically isomorphic to the original one: \(\mathfrak D(\mathfrak M(\mathbf D))\cong\mathbf D\).
\end{enumerate}
\end{theorem}

\section{A categorical equivalence: strict block morphisms}\label{sec:categorical-equivalence}

The next two sections recast the representation theorem in categorical language, first in a strict block form and then in an intrinsic \(\tau\)-compatible form.
The first formulation records the object-level reconstruction theorem as an equivalence of categories for strict block morphisms: each canonical component is sent into a canonical component, and the distinguished block unit is preserved.
Thus this section gives the cleanest categorical expression of the object-level reconstruction theorem.
The \(\tau\)-compatible formulation developed afterwards is less rigid on block units and more intrinsic to the local-unit map.

\begin{definition}[The strict block category]\label{def:cat-lua-blk-finite}
Let \(\mathbf{LUA}^{\mathrm{blk}}_{\mathrm{fin}}\) have as objects all finite local-unit-aligned totally ordered monoids.

Let \(\mathbf M\) and \(\mathbf N\) be such monoids.
A \emph{strict block morphism} \(f\colon \mathbf M\to \mathbf N\) is an isotone unital monoid homomorphism such that, for every \(A\in \Cm(\mathbf M)\), there exists a unique block \(\sigma_f(A)\in \Cm(\mathbf N)\) satisfying
\[
f(M_A)\subseteq N_{\sigma_f(A)}
\qquad\text{and}\qquad
f(e_A)=e_{\sigma_f(A)}.
\]
Composition is ordinary composition of maps.
\end{definition}

The uniqueness of the block is automatic once existence holds, since \(M_A\) is nonempty and the canonical components of \(\mathbf N\) are pairwise disjoint.

\begin{lemma}\label{lem:block-preserving-morphism-induces-sigma}
Let \(f\colon \mathbf M\to \mathbf N\) be a strict block morphism.
Then the induced map \(\sigma_f\colon \Cm(\mathbf M)\to \Cm(\mathbf N)\) is isotone.
\end{lemma}

\begin{proof}
Let \(A,B\in \Cm(\mathbf M)\) with \(A\le_{\min} B\).
By Lemma~\ref{lem:block-join-min}, this means \(e_A\le e_B\).
Since \(f\) is isotone, \(f(e_A)\le f(e_B)\).
By Definition~\ref{def:cat-lua-blk-finite}, \(f(e_A)=e_{\sigma_f(A)}\) and \(f(e_B)=e_{\sigma_f(B)}\).
Applying Lemma~\ref{lem:block-join-min} again, we obtain \(\sigma_f(A)\le_{\min}\sigma_f(B)\).
Thus \(\sigma_f\) is isotone.
\end{proof}

\begin{proposition}\label{prop:lua-blk-category}
The objects and morphisms described in Definition~\ref{def:cat-lua-blk-finite} form a category.
We denote it by \(\mathbf{LUA}^{\mathrm{blk}}_{\mathrm{fin}}\).
\end{proposition}

\begin{proof}
The identity map on any finite local-unit-aligned totally ordered monoid is plainly an isotone unital monoid homomorphism, and it preserves each canonical block and each block unit.
Thus identities are morphisms.

Let \(f\colon \mathbf M\to \mathbf N\) and \(g\colon \mathbf N\to \mathbf P\) be morphisms.
For each \(A\in \Cm(\mathbf M)\), one has \(f(M_A)\subseteq N_{\sigma_f(A)}\), and hence
\[
g(f(M_A))\subseteq P_{\sigma_g(\sigma_f(A))}.
\]
Also,
\[
(g\circ f)(e_A)
=
g(e_{\sigma_f(A)})
=
e_{\sigma_g(\sigma_f(A))}.
\]
Therefore \(g\circ f\) is again a morphism in the sense of Definition~\ref{def:cat-lua-blk-finite}.
The associativity and identity axioms are inherited from ordinary composition of maps.
\end{proof}

\begin{definition}[Rigid direct systems]\label{def:rigid-direct-system-object}
A \emph{rigid direct system} is a unit-constant finite chain-indexed direct system
\[
\mathbf D=\bigl((\mathbf M_i)_{i\in I},
(\rho^{\mathbf D}_{i,k})_{i\le k}\bigr)
\]
whose constituent monoids are \(\tau\)-multiplication-cohesive finite local-unit-aligned totally ordered monoids.
\end{definition}

\begin{definition}[Rigid-system morphisms]\label{def:rigid-direct-system-morphism}
Let
\[
\mathbf D=\bigl((\mathbf M_i)_{i\in I},(\rho^{\mathbf D}_{i,k})_{i\le k}\bigr)
\qquad\text{and}\qquad
\mathbf E=\bigl((\mathbf N_j)_{j\in J},(\rho^{\mathbf E}_{j,\ell})_{j\le \ell}\bigr)
\]
be rigid direct systems, and write \(i_0:=\min I\) and \(j_0:=\min J\).
A \emph{morphism} \(\Phi\colon \mathbf D\to \mathbf E \) is a pair \(\Phi=(\sigma,(\Phi_i)_{i\in I})\) consisting of the following items:
\begin{enumerate}
\taggeditem{\textup{(DS1)}}{item:rigds-index}
an isotone map \(\sigma\colon I\to J\) with \(\sigma(i_0)=j_0\);

\taggeditem{\textup{(DS2)}}{item:rigds-components}
for each \(i\in I\), an isotone unital monoid homomorphism \(\Phi_i\colon \mathbf M_i\to \mathbf N_{\sigma(i)}\);

\taggeditem{\textup{(DS3)}}{item:rigds-transition-compatibility}
for all \(i\le k\) in \(I\), the compatibility relation
\[
\Phi_k\circ \rho^{\mathbf D}_{i,k}
=
\rho^{\mathbf E}_{\sigma(i),\sigma(k)}\circ \Phi_i
\]
holds;

\taggeditem{\textup{(Ord)}}{item:rigds-directed-order-condition}
for all \(i\le k\) in \(I\), all \(x\in M_k\), and all \(y\in M_i\), if \(x <_k \rho^{\mathbf D}_{i,k}(y)\), then
\[
\begin{cases}
\Phi_k(x)\le_{\sigma(i)} \Phi_i(y), & \text{if }\sigma(i)=\sigma(k),\\[1mm]
\Phi_k(x)<_{\sigma(k)}
\rho^{\mathbf E}_{\sigma(i),\sigma(k)}\!\bigl(\Phi_i(y)\bigr),
& \text{if }\sigma(i)<\sigma(k).
\end{cases}
\]
\end{enumerate}
The conditions \ref{item:rigds-index}--\ref{item:rigds-transition-compatibility} are the componentwise direct-system compatibility conditions associated with the chosen isotone index map \(\sigma\).
We do not require \(\sigma\) to be injective: this is deliberate, since non-injectivity records the collapse of several source components into one target component.
Thus this is a broader system-morphism convention than the injective-index convention often used for direct-system embeddings in ordered algebra.
Condition \ref{item:rigds-directed-order-condition} is the additional requirement ensuring that the piecewise map preserves the directed lexicographic order in the case where a higher source component is compared with a lower one, namely the reversed cross-level case.

The identity morphism on \(\mathbf D\) is \(\bigl(\mathrm{id}_I,(\mathrm{id}_{\mathbf M_i})_{i\in I}\bigr), \) and if
\[
\Phi=(\sigma,(\Phi_i)_{i\in I})\colon \mathbf D\to \mathbf E,
\qquad
\Psi=(\nu,(\Psi_j)_{j\in J})\colon \mathbf E\to \mathbf F,
\]
then their composite is declared to be
\[
\Psi\circ \Phi
:=
\bigl(\nu\circ\sigma,(\Psi_{\sigma(i)}\circ \Phi_i)_{i\in I}\bigr).
\]
\end{definition}

\begin{remark}\label{rem:collapse-via-sigma}
Because \(\sigma\) is only required to be isotone, several source blocks may be sent to the same target block.
Thus if \(i<k\) and \(\sigma(i)=\sigma(k)\), then condition \ref{item:rigds-transition-compatibility} becomes \(\Phi_k\circ \rho^{\mathbf D}_{i,k}=\Phi_i. \)
Since \(\mathbf D\) is rigid, \(\rho^{\mathbf D}_{i,k}\) is unit-constant, so \(\Phi_i\) is forced to be unit-constant as well.
Hence the non-injective index map \(\sigma\) encodes precisely the collapse phenomenon: blocks lying in the same fibre of \(\sigma\) are merged.
\end{remark}

\begin{proposition}\label{prop:realization-of-rigds-morphism}
Let
\[
\Phi=\bigl(\sigma,(\Phi_i)_{i\in I}\bigr)\colon \mathbf D\to \mathbf E
\]
be a morphism of rigid direct systems.
Its \emph{ambient realization} is the piecewise-defined map
\[
\mathfrak M(\Phi)\colon \mathfrak M(\mathbf D)\to \mathfrak M(\mathbf E),
\qquad
\mathfrak M(\Phi)|_{M_i}:=\Phi_i.
\]
Then \(\mathfrak M(\Phi)\) is a morphism in \(\mathbf{LUA}^{\mathrm{blk}}_{\mathrm{fin}}\).
More precisely, it is an isotone unital monoid homomorphism and, under the canonical identifications of construction components with canonical blocks, its induced block map is the index map \(\sigma\).
\end{proposition}

\begin{proof}
Since the summands \(M_i\) are pairwise disjoint, the map is well defined.
Because \(\sigma(i_0)=j_0\) and \(\Phi_{i_0}\) is unital, the global unit \(e_{i_0}\) of \(\mathfrak M(\mathbf D)\) is sent to the global unit \(e_{j_0}\) of \(\mathfrak M(\mathbf E)\).

We first prove multiplicativity.
Let \(x\in M_i\) and \(y\in M_k\), and put \(m:=\max\{i,k\}\) and \(n:=\max\{\sigma(i),\sigma(k)\}=\sigma(m)\), where the last equality holds because \(\sigma\) is isotone.
Then in \(\mathfrak M(\mathbf D)\), \(xy=\rho^{\mathbf D}_{i,m}(x)\cdot_m \rho^{\mathbf D}_{k,m}(y)\in M_m\).
Hence
\begin{align*}
\mathfrak M(\Phi)(xy)
&=\Phi_m\bigl(\rho^{\mathbf D}_{i,m}(x)\cdot_m
\rho^{\mathbf D}_{k,m}(y)\bigr)\\
&=\Phi_m\bigl(\rho^{\mathbf D}_{i,m}(x)\bigr)\cdot_n
\Phi_m\bigl(\rho^{\mathbf D}_{k,m}(y)\bigr)\\
&=\rho^{\mathbf E}_{\sigma(i),n}(\Phi_i(x))\cdot_n
\rho^{\mathbf E}_{\sigma(k),n}(\Phi_k(y)),
\end{align*}
by condition \ref{item:rigds-transition-compatibility} of Definition~\ref{def:rigid-direct-system-morphism}.
This is exactly the ambient product of \(\Phi_i(x)\) and \(\Phi_k(y)\) in \(\mathfrak M(\mathbf E)\).
Therefore \(\mathfrak M(\Phi)(xy)=\mathfrak M(\Phi)(x)\mathfrak M(\Phi)(y)\).

We now prove isotonicity.
Let \(x\in M_i\), \(y\in M_k\), and assume \(x\le y\) in \(\mathfrak M(\mathbf D)\).

If \(i=k\), then \(x\le_i y\), so isotonicity of \(\Phi_i\) gives \(\mathfrak M(\Phi)(x)=\Phi_i(x)\le_{\sigma(i)} \Phi_i(y)=\mathfrak M(\Phi)(y)\).

Assume next that \(i<k\).
By Remark~\ref{rem:directed-lex-order-working-forms}, we have \(\rho^{\mathbf D}_{i,k}(x)\le_k y\).
Applying the isotone map \(\Phi_k\) and using condition \ref{item:rigds-transition-compatibility}, we obtain
\[
\rho^{\mathbf E}_{\sigma(i),\sigma(k)}(\Phi_i(x))
=
\Phi_k\bigl(\rho^{\mathbf D}_{i,k}(x)\bigr)
\le_{\sigma(k)}
\Phi_k(y).
\]
Since \(\sigma(i)\le \sigma(k)\), another application of Remark~\ref{rem:directed-lex-order-working-forms} shows that \(\mathfrak M(\Phi)(x)\le \mathfrak M(\Phi)(y)\).

Finally, assume that \(k<i\).
By Remark~\ref{rem:directed-lex-order-working-forms}, the inequality \(x\le y\) implies \(x<_i \rho^{\mathbf D}_{k,i}(y)\).
Applying condition \ref{item:rigds-directed-order-condition} with \(k\le i\), \(x\in M_i\), and \(y\in M_k\), we obtain either \(\Phi_i(x)\le_{\sigma(i)}\Phi_k(y)\) if \(\sigma(k)=\sigma(i)\), or
\[
\Phi_i(x)<_{\sigma(i)}
\rho^{\mathbf E}_{\sigma(k),\sigma(i)}\!\bigl(\Phi_k(y)\bigr)
\]
if \(\sigma(k)<\sigma(i)\).
In the first case the two images lie in the same component of \(\mathfrak M(\mathbf E)\), so the displayed internal inequality is exactly \(\mathfrak M(\Phi)(x)\le \mathfrak M(\Phi)(y)\).
In the second case, Remark~\ref{rem:directed-lex-order-working-forms} gives \(\mathfrak M(\Phi)(x)<\mathfrak M(\Phi)(y)\).
Thus \(\mathfrak M(\Phi)\) is isotone.

It remains only to verify the strict block condition.
For each \(i\in I\), put
\[
P_i:=\PosId(\mathfrak M(\mathbf D))\cap M_i,
\qquad
Q_{\sigma(i)}:=\PosId(\mathfrak M(\mathbf E))\cap N_{\sigma(i)}.
\]
By Theorem~\ref{thm:representation}\ref{item:representation-direct-system-recovery}, these are canonical blocks, with \(\layer{P_i}=M_i\) and \(\layer{Q_{\sigma(i)}}=N_{\sigma(i)}\), and their block units are respectively \(e_i\) and \(e_{\sigma(i)}\).
Since \(\mathfrak M(\Phi)(M_i)=\Phi_i(M_i)\subseteq N_{\sigma(i)}\) and
\[
\mathfrak M(\Phi)(e_i)=\Phi_i(e_i)=e_{\sigma(i)},
\]
the map \(\mathfrak M(\Phi)\) sends the canonical block \(P_i\) into \(Q_{\sigma(i)}\) and sends its block unit to the block unit of \(Q_{\sigma(i)}\).
Hence \(\mathfrak M(\Phi)\) is a morphism in \(\mathbf{LUA}^{\mathrm{blk}}_{\mathrm{fin}}\), and its induced block map corresponds to \(\sigma\) under the canonical identifications \(i\leftrightarrow P_i\) and \(j\leftrightarrow Q_j\).
\end{proof}

\begin{proposition}\label{prop:ambient-maps-vs-rigds-morphisms}
Let \(\mathbf D\) and \(\mathbf E\) be rigid direct systems.
The ambient realization construction of Proposition~\ref{prop:realization-of-rigds-morphism} gives a bijection
\[
\operatorname{Hom}_{\mathbf{RigDS}_{\mathrm{fin}}}(\mathbf D,\mathbf E)
\cong
\operatorname{Hom}_{\mathbf{LUA}^{\mathrm{blk}}_{\mathrm{fin}}}
(\mathfrak M(\mathbf D),\mathfrak M(\mathbf E)).
\]
More explicitly, a rigid-system morphism \(\Phi\) is sent to its piecewise realization \(\mathfrak M(\Phi)\), and every strict block morphism \(f\colon \mathfrak M(\mathbf D)\to \mathfrak M(\mathbf E)\) is realized by a unique rigid-system morphism.
\end{proposition}

\begin{proof}
By Proposition~\ref{prop:realization-of-rigds-morphism}, every morphism \(\Phi\colon\mathbf D\to\mathbf E\) of rigid direct systems has a piecewise realization \(\mathfrak M(\Phi)\), and this realization is a morphism in \(\mathbf{LUA}^{\mathrm{blk}}_{\mathrm{fin}}\).
If \(g\colon \mathfrak M(\mathbf D)\to \mathfrak M(\mathbf E)\) is any ambient realization of \(\Phi\), then for each \(i\in I\) one has \(g|_{M_i}=\Phi_i\).
Since the blocks \(M_i\) are pairwise disjoint and their union is all of \(\mathfrak M(\mathbf D)\), these restrictions determine \(g\) uniquely.

Now let
\[
f\colon \mathfrak M(\mathbf D)\to \mathfrak M(\mathbf E)
\]
be a morphism in \(\mathbf{LUA}^{\mathrm{blk}}_{\mathrm{fin}}\).
For each \(i\in I\), put
\[
P_i:=\PosId(\mathfrak M(\mathbf D))\cap M_i.
\]
By Theorem~\ref{thm:representation}\ref{item:representation-direct-system-recovery}, the \(P_i\)'s are the canonical blocks of \(\mathfrak M(\mathbf D)\), and \(\layer{P_i}=M_i\).
Similarly, the canonical blocks of \(\mathfrak M(\mathbf E)\) are
\[
Q_j:=\PosId(\mathfrak M(\mathbf E))\cap N_j
\qquad(j\in J),
\]
with \(\layer{Q_j}=N_j\).

Since \(f\) is a strict block morphism, for each \(i\in I\) there is a unique \(j\in J\) such that
\[
f(M_i)\subseteq N_j
\qquad\text{and}\qquad
f(e_i)=e_j.
\]
Define \(\sigma(i):=j\).
Equivalently, \(Q_{\sigma(i)}\) is the canonical block \(\sigma_f(P_i)\) of \(\mathfrak M(\mathbf E)\).
By Lemma~\ref{lem:block-preserving-morphism-induces-sigma}, the block map \(\sigma_f\) is isotone; since the maps \(i\mapsto P_i\) and \(j\mapsto Q_j\) are order isomorphisms by Theorem~\ref{thm:representation}\ref{item:representation-direct-system-recovery}, the induced map \(\sigma\colon I\to J\) is isotone.

For each \(i\in I\), define
\[
\Phi_i:=f|_{M_i}\colon \mathbf M_i\to \mathbf N_{\sigma(i)}.
\]
Since \(f(M_i)\subseteq N_{\sigma(i)}\) and \(f\) is isotone and multiplicative, each \(\Phi_i\) is an isotone monoid homomorphism.
It is unital because \(\Phi_i(e_i)=f(e_i)=e_{\sigma(i)}\).
Also, since \(f\) is unital and \(e_{i_0}\), \(e_{j_0}\) are the global units of \(\mathfrak M(\mathbf D)\) and \(\mathfrak M(\mathbf E)\), we have \(e_{j_0}=f(e_{i_0})=e_{\sigma(i_0)}\), hence \(\sigma(i_0)=j_0\).

We verify condition \ref{item:rigds-transition-compatibility}.
Let \(i\le k\) and \(x\in M_i\).
In \(\mathfrak M(\mathbf D)\), one has \(\rho^{\mathbf D}_{i,k}(x)=xe_k\).
Therefore
\begin{align*}
\Phi_k\bigl(\rho^{\mathbf D}_{i,k}(x)\bigr)
&=
f\bigl(\rho^{\mathbf D}_{i,k}(x)\bigr)
=
f(xe_k)
=
f(x)f(e_k)\\
&=
\Phi_i(x)e_{\sigma(k)}
=
\rho^{\mathbf E}_{\sigma(i),\sigma(k)}\!\bigl(\Phi_i(x)\bigr),
\end{align*}
where the last equality holds because \(\sigma(i)\le\sigma(k)\) and the ambient product with the unit \(e_{\sigma(k)}\) of \(N_{\sigma(k)}\) realizes the transition map \(\rho^{\mathbf E}_{\sigma(i),\sigma(k)}\).

It remains to verify condition \ref{item:rigds-directed-order-condition}.
Let \(i\le k\), let \(x\in M_k\), and let \(y\in M_i\) satisfy \(x<_k\rho^{\mathbf D}_{i,k}(y)\).
By Remark~\ref{rem:directed-lex-order-working-forms}, this is equivalent to \(x<y\) in \(\mathfrak M(\mathbf D)\).
Since \(f\) is isotone, \(f(x)\le f(y)\) in \(\mathfrak M(\mathbf E)\).

If \(\sigma(i)=\sigma(k)\), then both \(f(x)\) and \(f(y)\) lie in the same block \(N_{\sigma(i)}\), so this means
\[
\Phi_k(x)\le_{\sigma(i)}\Phi_i(y).
\]
If \(\sigma(i)<\sigma(k)\), then \(f(x)\in N_{\sigma(k)}\) and \(f(y)\in N_{\sigma(i)}\).
By Remark~\ref{rem:directed-lex-order-working-forms}, the inequality \(f(x)\le f(y)\) is equivalent to
\[
\Phi_k(x)<_{\sigma(k)}
\rho^{\mathbf E}_{\sigma(i),\sigma(k)}\!\bigl(\Phi_i(y)\bigr).
\]
Thus condition \ref{item:rigds-directed-order-condition} holds.

Hence
\[
\Phi:=\bigl(\sigma,(\Phi_i)_{i\in I}\bigr)
\]
is a morphism of rigid direct systems.
By construction, \(\mathfrak M(\Phi)=f\).
To prove uniqueness, let \(\Psi=(\nu,(\Psi_i)_{i\in I})\colon \mathbf D\to \mathbf E\) be another morphism with \(\mathfrak M(\Psi)=f\).
Then \(\Psi_i=f|_{M_i}=\Phi_i\) for every \(i\in I\), and
\[
e_{\nu(i)}
=
\Psi_i(e_i)
=
f(e_i)
=
e_{\sigma(i)}.
\]
Since the block units in \(\mathfrak M(\mathbf E)\) are pairwise distinct, it follows that \(\nu(i)=\sigma(i)\) for every \(i\in I\).
Hence \(\Psi=\Phi\).
\end{proof}

\begin{proposition}\label{prop:rigds-category}
The objects of Definition~\ref{def:rigid-direct-system-object} and the morphisms of Definition~\ref{def:rigid-direct-system-morphism}, equipped with the identity morphisms and composition law specified there, form a category.
We denote it by \(\mathbf{RigDS}_{\mathrm{fin}}\).
\end{proposition}

\begin{proof}
Let \(\mathbf D=\bigl((\mathbf M_i)_{i\in I},(\rho^{\mathbf D}_{i,k})_{i\le k}\bigr)\) be a rigid direct system.
The pair
\[
\bigl(\mathrm{id}_I,(\mathrm{id}_{\mathbf M_i})_{i\in I}\bigr)
\]
satisfies the defining conditions of Definition~\ref{def:rigid-direct-system-morphism}: the index map is isotone, preserves the least element, the component maps are isotone unital monoid homomorphisms, and the transition-compatibility and directed-order conditions are immediate.
Hence it is a morphism \(\mathbf D\to\mathbf D\).

Now let
\[
\Phi=(\sigma,(\Phi_i)_{i\in I})\colon \mathbf D\to \mathbf E,
\qquad
\Psi=(\nu,(\Psi_j)_{j\in J})\colon \mathbf E\to \mathbf F
\]
be morphisms of rigid direct systems, where
\[
\mathbf F=
\bigl((\mathbf P_\ell)_{\ell\in K},
(\rho^{\mathbf F}_{\ell,m})_{\ell\le m}\bigr).
\]
Put \(h:=\mathfrak M(\Psi)\circ \mathfrak M(\Phi)\).
By Proposition~\ref{prop:realization-of-rigds-morphism}, \(h\) is an isotone unital monoid homomorphism \(\mathfrak M(\mathbf D)\to \mathfrak M(\mathbf F)\).
For each \(i\in I\), one has
\[
h(M_i)\subseteq P_{\nu(\sigma(i))}
\qquad\text{and}\qquad
h(e_i)=e_{\nu(\sigma(i))}.
\]
Thus \(h\) is a strict block morphism \(\mathfrak M(\mathbf D)\to\mathfrak M(\mathbf F)\).
Hence the inverse-realization part of Proposition~\ref{prop:ambient-maps-vs-rigds-morphisms} yields a unique morphism
\[
\Theta=\bigl(\nu\circ\sigma,(\Theta_i)_{i\in I}\bigr)\colon \mathbf D\to \mathbf F
\]
with \(\mathfrak M(\Theta)=h\).
For each \(i\in I\), the restriction of \(h\) to \(M_i\) is \(h|_{M_i}=\Psi_{\sigma(i)}\circ \Phi_i,\) so uniqueness of the block restrictions gives \(\Theta_i=\Psi_{\sigma(i)}\circ \Phi_i\).
Thus \(\Theta\) is exactly the composite declared in Definition~\ref{def:rigid-direct-system-morphism}, and that declared composite is a morphism.

Associativity follows immediately from the componentwise formula for composition, and the identity laws are immediate as well.
\end{proof}

\begin{theorem}[Categorical equivalence for strict block morphisms]\label{thm:categorical-equivalence}
The assignments
\[
\mathfrak D\colon \mathbf{LUA}^{\mathrm{blk}}_{\mathrm{fin}}\to \mathbf{RigDS}_{\mathrm{fin}},
\qquad
\mathfrak M\colon \mathbf{RigDS}_{\mathrm{fin}}\to \mathbf{LUA}^{\mathrm{blk}}_{\mathrm{fin}}
\]
form quasi-inverse equivalences of categories.
\end{theorem}

\begin{proof}
On objects, \(\mathfrak D\) and \(\mathfrak M\) are exactly those from Theorem~\ref{thm:representation}: \(\mathfrak D\) sends a finite local-unit-aligned totally ordered monoid to its canonical rigid direct system, and \(\mathfrak M\) reconstructs the ambient monoid from a rigid direct system.

On morphisms, \(\mathfrak M\) sends a rigid-system morphism \(\Phi\colon\mathbf D\to\mathbf E\) to its piecewise realization \(\mathfrak M(\Phi)\).
By Proposition~\ref{prop:realization-of-rigds-morphism}, this is a morphism in \(\mathbf{LUA}^{\mathrm{blk}}_{\mathrm{fin}}\).

Conversely, if \(f\colon \mathbf M\to \mathbf N\) is a morphism in \(\mathbf{LUA}^{\mathrm{blk}}_{\mathrm{fin}}\), then Proposition~\ref{prop:ambient-maps-vs-rigds-morphisms}, applied to the canonical rigid direct systems \(\mathfrak D(\mathbf M)\) and \(\mathfrak D(\mathbf N)\), gives a unique rigid-system morphism
\[
\mathfrak D(f)\colon \mathfrak D(\mathbf M)\to \mathfrak D(\mathbf N)
\]
whose ambient realization is \(f\).
This defines \(\mathfrak D\) on morphisms.

Functoriality of \(\mathfrak M\) follows from the componentwise definition of composition in \(\mathbf{RigDS}_{\mathrm{fin}}\).
For \(\mathfrak D\), let \(f\colon \mathbf M\to \mathbf N\) and \(g\colon \mathbf N\to \mathbf P\) be morphisms in \(\mathbf{LUA}^{\mathrm{blk}}_{\mathrm{fin}}\).
Then
\begin{align*}
\mathfrak M\bigl(\mathfrak D(g)\circ \mathfrak D(f)\bigr)
&=\mathfrak M(\mathfrak D(g))\circ \mathfrak M(\mathfrak D(f))\\
&=g\circ f
 =\mathfrak M(\mathfrak D(g\circ f)).
\end{align*}
By the uniqueness clause in Proposition~\ref{prop:ambient-maps-vs-rigds-morphisms}, \(\mathfrak D(g\circ f)=\mathfrak D(g)\circ \mathfrak D(f)\).
The identity case is similar.

We now verify that the two functors are quasi-inverse.

Let \(\mathbf M\in \mathbf{LUA}^{\mathrm{blk}}_{\mathrm{fin}}\).
By Theorem~\ref{thm:representation}\ref{item:representation-canonical-recovery}, there is a canonical isomorphism of ambient monoids
\[
\eta_{\mathbf M}\colon \mathfrak M(\mathfrak D(\mathbf M))
\xrightarrow{\ \cong\ } \mathbf M.
\]
Here \(\eta_{\mathbf M}\) is induced by the identity maps on the canonical component monoids: in the tagged-disjoint-copy convention, it sends the copy of \(x\in M_A\) occurring in the reconstructed union to the element \(x\) of \(\mathbf M\).
Hence \(\eta_{\mathbf M}\) sends the reconstructed block over \(A\) onto \(M_A\) and sends its block unit to \(e_A\), so \(\eta_{\mathbf M}\) is a morphism in \(\mathbf{LUA}^{\mathrm{blk}}_{\mathrm{fin}}\).
If \(f\colon \mathbf M\to \mathbf N\) is such a morphism, then the typed equality
\[
\eta_{\mathbf N}\circ \mathfrak M(\mathfrak D(f))
=f\circ \eta_{\mathbf M}
\]
holds on every reconstructed canonical component.
Thus the family \((\eta_{\mathbf M})\) is a natural isomorphism \(\mathfrak M\circ \mathfrak D\Rightarrow \mathrm{Id}_{\mathbf{LUA}^{\mathrm{blk}}_{\mathrm{fin}}}\).

Now let \(\mathbf D=\bigl((\mathbf M_i)_{i\in I},(\rho^{\mathbf D}_{i,k})_{i\le k}\bigr)\) be a rigid direct system.
By Theorem~\ref{thm:representation}\ref{item:representation-direct-system-recovery}, the canonical decomposition of \(\mathfrak M(\mathbf D)\) coincides with the given construction-level layers \(M_i\).
More precisely, if \(P_i:=\PosId(\mathfrak M(\mathbf D))\cap M_i, \) then
\[
\Cm(\mathfrak M(\mathbf D))=\{P_i:i\in I\}
\qquad\text{and}\qquad
\layer{P_i}=M_i.
\]
Let
\[
\kappa_{\mathbf D}\colon I\to \Cm(\mathfrak M(\mathbf D)),
\qquad
\kappa_{\mathbf D}(i):=P_i.
\]
This map is bijective.
If \(i\le k\), then \(e_i\le e_k\) in the directed lexicographic order, and hence \(P_i\le_{\min}P_k\) by Lemma~\ref{lem:block-join-min}.
Conversely, suppose \(P_i\le_{\min}P_k\).
Then \(e_i\le e_k\) in \(\mathfrak M(\mathbf D)\).
If \(k<i\), the reversed working form of the directed lexicographic order would give \(e_i<_i\rho^{\mathbf D}_{k,i}(e_k)\).
Since \(\mathbf D\) is rigid, \(\rho^{\mathbf D}_{k,i}(e_k)=e_i\), a contradiction.
Hence \(i\le k\).
Thus \(\kappa_{\mathbf D}\) is an order isomorphism.

Let \(\alpha_{\mathbf D}:=\eta_{\mathfrak M(\mathbf D)}^{-1}\colon \mathfrak M(\mathbf D)\to \mathfrak M(\mathfrak D(\mathfrak M(\mathbf D)))\).
For each \(i\in I\), the map \(\alpha_{\mathbf D}\) sends the construction-level component \(M_i\) into the reconstructed canonical component over \(P_i\), and sends \(e_i\) to the block unit of that component.
Therefore Proposition~\ref{prop:ambient-maps-vs-rigds-morphisms} yields a unique morphism
\[
\varepsilon_{\mathbf D}:=
\bigl(\kappa_{\mathbf D},(\mathrm{id}_{\mathbf M_i})_{i\in I}\bigr)
\colon \mathbf D\to \mathfrak D(\mathfrak M(\mathbf D))
\]
such that \(\mathfrak M(\varepsilon_{\mathbf D})=\alpha_{\mathbf D}\).
Likewise let \(\beta_{\mathbf D}:=\eta_{\mathfrak M(\mathbf D)}\colon \mathfrak M(\mathfrak D(\mathfrak M(\mathbf D)))\to \mathfrak M(\mathbf D)\).
For each \(i\in I\), the map \(\beta_{\mathbf D}\) sends the reconstructed canonical component indexed by \(P_i=\kappa_{\mathbf D}(i)\) onto the original construction-level component \(M_i\), and sends its block unit to \(e_i\).
Thus Proposition~\ref{prop:ambient-maps-vs-rigds-morphisms} yields a unique morphism
\[
\delta_{\mathbf D}\colon \mathfrak D(\mathfrak M(\mathbf D))\to \mathbf D
\]
such that \(\mathfrak M(\delta_{\mathbf D})=\beta_{\mathbf D}\).
Its induced index map is \(\kappa_{\mathbf D}^{-1}\), since the component over \(P_i\) is mapped onto \(M_i\).
Hence
\begin{align*}
\mathfrak M(\delta_{\mathbf D}\circ \varepsilon_{\mathbf D})
&=\beta_{\mathbf D}\circ\alpha_{\mathbf D}
 =\mathrm{id}_{\mathfrak M(\mathbf D)}
 =\mathfrak M(\mathrm{id}_{\mathbf D}),
\end{align*}
and similarly
\[
\mathfrak M(\varepsilon_{\mathbf D}\circ \delta_{\mathbf D})
=\alpha_{\mathbf D}\circ\beta_{\mathbf D}
=\mathrm{id}_{\mathfrak M(\mathfrak D(\mathfrak M(\mathbf D)))}
=\mathfrak M(\mathrm{id}_{\mathfrak D(\mathfrak M(\mathbf D))}).
\]
By the uniqueness clause in Proposition~\ref{prop:ambient-maps-vs-rigds-morphisms}, it follows that \(\delta_{\mathbf D}\circ \varepsilon_{\mathbf D}=\mathrm{id}_{\mathbf D}\) and \(\varepsilon_{\mathbf D}\circ \delta_{\mathbf D} =\mathrm{id}_{\mathfrak D(\mathfrak M(\mathbf D))}\).
Thus \(\varepsilon_{\mathbf D}\) is an isomorphism in \(\mathbf{RigDS}_{\mathrm{fin}}\).

To prove naturality, let \(\Phi\colon \mathbf D\to \mathbf E\) be a morphism in \(\mathbf{RigDS}_{\mathrm{fin}}\).
By the already proved naturality of \(\eta\) on the monoid side, the two ambient realizations satisfy
\[
\mathfrak M\bigl(\mathfrak D(\mathfrak M(\Phi))\circ \varepsilon_{\mathbf D}\bigr)
=
\eta_{\mathfrak M(\mathbf E)}^{-1}\circ \mathfrak M(\Phi)
=
\mathfrak M\bigl(\varepsilon_{\mathbf E}\circ \Phi\bigr).
\]
Both morphisms \(\mathfrak D(\mathfrak M(\Phi))\circ \varepsilon_{\mathbf D}\) and \(\varepsilon_{\mathbf E}\circ \Phi\) therefore realize the same typed ambient map.
By the uniqueness clause in Proposition~\ref{prop:ambient-maps-vs-rigds-morphisms}, they are equal.
Thus the family \((\varepsilon_{\mathbf D})\) is a natural isomorphism \(\mathrm{Id}_{\mathbf{RigDS}_{\mathrm{fin}}}\Rightarrow \mathfrak D\circ \mathfrak M\).

Therefore \(\mathfrak D\) and \(\mathfrak M\) are quasi-inverse equivalences of categories.
\end{proof}

\begin{example}\label{ex:block-preserving-not-unit-preserving}
In Definition~\ref{def:cat-lua-blk-finite}, the block-unit condition \(f(e_A)=e_{\sigma_f(A)}\) is not implied by the block-only requirement \(f(M_A)\subseteq N_{\sigma_f(A)}\) for \(A\in \Cm(\mathbf M)\).
Indeed, let \(\mathbf S:=\langle \{1,2\},\le,\max,1\rangle\).
Then \(\PosId(\mathbf S)=\{1,2\}\), \(\tau_{\mathbf S}(1)=1\), and \(\tau_{\mathbf S}(2)=2\).
Hence \(\Cm(\mathbf S)=\bigl\{\{1\},\{2\}\bigr\}\).
For \(A:=\{2\}\), the canonical block is \(M_A=\{2\}\) and its block unit is \(e_A=2\).

Let \(\mathbf T\) be the totally ordered monoid on \(\{1<2<3<4\}\), with identity \(1\), whose multiplication and local-unit map are
\[
\begin{array}{c|cccc}
\cdot &1&2&3&4\\ \hline
1&1&2&3&4\\
2&2&2&4&4\\
3&3&4&4&4\\
4&4&4&4&4
\end{array}
\qquad
\tau_{\mathbf T}(1)=1,\quad \tau_{\mathbf T}(2)=2,\quad
\tau_{\mathbf T}(3)=1,\quad \tau_{\mathbf T}(4)=4.
\]
A direct check gives \(\Cm(\mathbf T)=\bigl\{\{1,2,4\}\bigr\}\).
Write \(B:=\{1,2,4\}\).
Then \(N_B=T\) and \(e_B=1\).
Define \(f\colon \mathbf S\to \mathbf T\) by \(f(1):=1\) and \(f(2):=2\).
This map is isotone, unital, and multiplicative, because the restriction of the multiplication of \(\mathbf T\) to \(\{1,2\}\) is the max-monoid multiplication.
Since \(\mathbf T\) has only one canonical block, \(f\) automatically satisfies the block-only condition.
In particular, \(f(M_A)=\{2\}\subseteq N_B\).
However, \(f(e_A)=f(2)=2\ne 1=e_B\).
Thus the block-only condition does not force preservation of block units.
\end{example}

\section{A second categorical equivalence: \texorpdfstring{\(\tau\)}{tau}-compatible homomorphisms}
\label{sec:taucomp-categorical-equivalence}

The second categorical formulation is the intrinsic counterpart of the strict block one.
Here the morphisms are not required to preserve the individual canonical block units; they are required only to commute with the local-unit map \(\tau\).
Such \(\tau\)-compatible homomorphisms may collapse distinct canonical components, but rigidity prevents them from splitting a source component among several target components.
This gives the categorical form of the representation for morphisms defined directly by the local-unit structure.

\begin{definition}[\(\tau\)-compatible homomorphism]
\label{def:taucomp-homomorphism}
Let \(\mathbf M\) and \(\mathbf N\) be finite local-unit-aligned totally ordered monoids.
An isotone unital monoid homomorphism \(f\colon \mathbf M\to \mathbf N \) is called \emph{\(\tau\)-compatible} if, for every \(x\in M\),
\[
\tau_{\mathbf N}(f(x))=f(\tau_{\mathbf M}(x)).
\]
\end{definition}

\begin{lemma}\label{lem:taucomp-layer-preservation}
Let \(f\colon \mathbf M\to \mathbf N\) be a \(\tau\)-compatible homomorphism.
Then:
\begin{enumerate}
\item\label{item:positiveLESZ} if \(u\in\PosId(\mathbf M)\), then \(f(u)\in\PosId(\mathbf N)\);
\item for every \(u\in\PosId(\mathbf M)\), \(f[\layer{u}^{\mathbf M}]\subseteq \layer{f(u)}^{\mathbf N}; \)
\item \(\tau\)-compatible homomorphisms are closed under composition.
\end{enumerate}
\end{lemma}

\begin{proof}
If \(u\in\PosId(\mathbf M)\), then \(u^2=u\) and \(u\ge e_{\mathbf M}\).
Hence \(f(u)^2=f(u)\), and, since \(f\) is isotone and unital, \(f(u)\ge f(e_{\mathbf M})=e_{\mathbf N}\).
Thus \(f(u)\) is a positive idempotent of \(\mathbf N\).

If \(x\in \layer{u}^{\mathbf M}\), then \(\tau_{\mathbf M}(x)=u\).
By \(\tau\)-compatibility, \(\tau_{\mathbf N}(f(x))=f(\tau_{\mathbf M}(x))=f(u),\) so \(f(x)\in \layer{f(u)}^{\mathbf N}\).

Finally, if \(f\colon\mathbf M\to\mathbf N\) and \(g\colon\mathbf N\to\mathbf P\) are \(\tau\)-compatible, then \(\tau_{\mathbf P}(g(f(x)))=g(\tau_{\mathbf N}(f(x))) =g(f(\tau_{\mathbf M}(x))),\) so \(g\circ f\) is \(\tau\)-compatible.
\end{proof}

\begin{remark}\label{rem:taucomp-idempotent-collapse}
\(\tau\)-compatibility preserves exact local-unit information, but it need not preserve it injectively.
Thus \(u\neq v\) in \(\PosId(\mathbf M)\) may satisfy \(f(u)=f(v)\).
In that case the exact layers \(\layer{u}^{\mathbf M}\) and \(\layer{v}^{\mathbf M}\) are both mapped into \(\layer{f(u)}^{\mathbf N}\).
The direct-system formulation below records the same phenomenon by allowing the component index map to be non-injective.
Of course, not every family of componentwise $\tau$-compatible maps over a non-injective index map gives a morphism: the piecewise map must still respect the directed-lexicographic order and the reconstructed multiplication.
\end{remark}

\begin{lemma}[Positive idempotents detect the component order]
\label{lem:positive-idempotents-detect-component-order}
Let \(\mathbf N\) be a finite local-unit-aligned totally ordered monoid, let \(C,D\in\Cm(\mathbf N)\), and let \(p\in C\) and \(q\in D\) be positive idempotents.
If \(p\le q\), then \(C\le_{\min}D. \)
\end{lemma}

\begin{proof}
If \(C=D\), there is nothing to prove.
Suppose, towards a contradiction, that \(D<_{\min}C\).
Since \(p\in N_C\) and \(q\in N_D\), Theorem~\ref{thm:canonical-direct-system}\eqref{item:direct-order-recovery} gives \(p<\rho^{\mathbf N}_{D,C}(q)\) inside the component \(\mathbf N_C\).
By Corollary~\ref{cor:canonical-transition-maps-unit-constant}, the transition map \(\rho^{\mathbf N}_{D,C}\) is unit-constant, so \(\rho^{\mathbf N}_{D,C}(q)=e_C\).
Hence \(p<e_C\).
But \(p\in C=\PosId(\mathbf N_C)\), by Lemma~\ref{lem:intrinsic-tau-on-component}, and the identity of \(\mathbf N_C\) is \(e_C=\min C\).
Therefore \(e_C\le p\), a contradiction.
Thus \(D<_{\min}C\) cannot occur, and since the component order is total, \(C\le_{\min}D\).
\end{proof}

\begin{proposition}[\(\tau\)-compatible homomorphisms collapse canonical components]
\label{prop:taucomp-homomorphisms-collapse-components}
Let \(f\colon \mathbf M\to \mathbf N \) be a \(\tau\)-compatible homomorphism of finite local-unit-aligned totally ordered monoids.
Then for every \(A\in\Cm(\mathbf M)\) there exists a unique block \(\sigma_f(A)\in\Cm(\mathbf N)\) such that \(f[A]\subseteq \sigma_f(A). \)
Consequently, \(f[M_A]\subseteq N_{\sigma_f(A)}. \)
Moreover, the induced map \(\sigma_f\colon \Cm(\mathbf M)\to\Cm(\mathbf N) \) is isotone.
\end{proposition}

\begin{proof}
Fix \(A\in\Cm(\mathbf M)\).
We first prove that all idempotents of \(A\) are sent into a single canonical block of \(\mathbf N\).

By Lemma~\ref{lem:taucomp-layer-preservation}\eqref{item:positiveLESZ}, each \(f(u)\), \(u\in A\), is a positive idempotent of \(\mathbf N\), and hence lies in a unique block of \(\Cm(\mathbf N)\).
Define
\[
\beta_A\colon A\to \Cm(\mathbf N)
\]
by letting \(\beta_A(u)\) be the unique canonical block of \(\mathbf N\) containing \(f(u)\).
Let \(\mathcal R_A\) be the partition of \(A\) into the nonempty fibres of \(\beta_A\), equivalently
\[
u\equiv v
\quad\Longleftrightarrow\quad
f(u)\text{ and }f(v)\text{ belong to the same block of }\Cm(\mathbf N).
\]
We show that \(\mathcal R_A\) is a \(\tau\)-multiplication-coherent partition of \(\PosId(\mathbf M_A)=A\), where the \(\tau\)-saturated products are computed in \(\mathbf M_A\).

Let \(U,V\in\mathcal R_A\), and let \(B_U,B_V\in\Cm(\mathbf N)\) be the unique blocks with \(f[U]\subseteq B_U\) and \(f[V]\subseteq B_V\).
First note that \((\tprod{U}{V})^{\mathbf M_A}\) is nonempty.
Indeed, choose \(u\in U\) and \(v\in V\).
Since \(u\) and \(v\) are positive idempotents of \(\mathbf M_A\), we have \(\tau_{\mathbf M_A}(u)=u\) and \(\tau_{\mathbf M_A}(v)=v\); hence \(\tau_{\mathbf M_A}(uv)\in(\tprod{U}{V})^{\mathbf M_A}\).

Take \(w\in(\tprod{U}{V})^{\mathbf M_A}\).
By Lemma~\ref{lem:intrinsic-tau-on-component}, the intrinsic and ambient \(\tau\)-saturated products agree on \(A\), so \(w=\tau_{\mathbf M}(xy)\) for some \(x,y\in M_A\) with \(\tau_{\mathbf M}(x)\in U\) and \(\tau_{\mathbf M}(y)\in V\).
By \(\tau\)-compatibility,
\[
\tau_{\mathbf N}(f(x))=f(\tau_{\mathbf M}(x))\in B_U,
\qquad
\tau_{\mathbf N}(f(y))=f(\tau_{\mathbf M}(y))\in B_V.
\]
Thus \(f(x)\in N_{B_U}\) and \(f(y)\in N_{B_V}\).
Since \(\Cm(\mathbf N)\) is \(\tau\)-multiplication-coherent, \(\tau_{\mathbf N}(f(x)f(y))\in B_U\vee B_V\).
Using multiplicativity and \(\tau\)-compatibility again,
\[
f(w)
=
f(\tau_{\mathbf M}(xy))
=
\tau_{\mathbf N}(f(xy))
=
\tau_{\mathbf N}(f(x)f(y))
\in B_U\vee B_V.
\]
Hence every \(w\in(\tprod{U}{V})^{\mathbf M_A}\) lies in the fibre \(\beta_A^{-1}(B_U\vee B_V)\).
Since this fibre contains such a \(w\), it is nonempty, and by definition it is a single \(\mathcal R_A\)-class.
Therefore \(\mathcal R_A\) is \(\tau\)-multiplication-coherent inside \(\mathbf M_A\).

By Corollary~\ref{cor:canonical-components-already-terminal-direct}, the component \(\mathbf M_A\) is \(\tau\)-multiplication-cohesive.
Thus its only \(\tau\)-multiplication-coherent partition is the trivial one.
Therefore \(\mathcal R_A=\{A\}\).
Hence \(f[A]\) is contained in a unique block of \(\Cm(\mathbf N)\), which we denote by \(\sigma_f(A)\).

Now let \(x\in M_A\).
Then \(\tau_{\mathbf M}(x)\in A\), so \(f(\tau_{\mathbf M}(x))\in\sigma_f(A)\).
By \(\tau\)-compatibility, \(\tau_{\mathbf N}(f(x))=f(\tau_{\mathbf M}(x))\in\sigma_f(A),\) and hence \(f(x)\in N_{\sigma_f(A)}\).
This proves \(f[M_A]\subseteq N_{\sigma_f(A)}\).

It remains to prove that \(\sigma_f\) is isotone.
Let \(A,B\in\Cm(\mathbf M)\) with \(A\le_{\min}B\).
Then \(e_A\le e_B\).
Since \(f\) is isotone, \(f(e_A)\le f(e_B).\)
Also \(e_A,e_B\in\PosId(\mathbf M)\), so Lemma~\ref{lem:taucomp-layer-preservation}\eqref{item:positiveLESZ} applied to \(f\) shows that \(f(e_A)\) and \(f(e_B)\) are positive idempotents of \(\mathbf N\).
By the definition of \(\sigma_f\), they satisfy \(f(e_A)\in\sigma_f(A)\) and \(f(e_B)\in\sigma_f(B)\).
Therefore Lemma~\ref{lem:positive-idempotents-detect-component-order} gives \(\sigma_f(A)\le_{\min}\sigma_f(B)\).
Thus \(\sigma_f\) is isotone.
\end{proof}

\begin{corollary}[Component restrictions of \(\tau\)-compatible homomorphisms]
\label{cor:taucomp-component-restrictions}
Let \(f\colon\mathbf M\to\mathbf N\) be \(\tau\)-compatible, let \(A\in\Cm(\mathbf M)\), and put \(B:=\sigma_f(A)\).
Then \(f_A:=f\!\upharpoonright_{M_A}\colon M_A\to N_B \) is isotone and multiplicative, and it commutes with the intrinsic local-unit maps:
\[
\tau_{\mathbf N_B}(f_A(x)) = f_A(\tau_{\mathbf M_A}(x))
\qquad (x\in M_A).
\]
Moreover, \(\PosId(\mathbf N_B)=B\), and \(f_A(e_A)=f(e_A)\in B. \)
Thus \(f_A\), as a map between the component monoids \(\mathbf M_A\) and \(\mathbf N_B\), need not be unital; it is unital if and only if \(f(e_A)=e_B\).
\end{corollary}

\begin{proof}
The inclusion \(f_A[M_A]\subseteq N_B\) is Proposition~\ref{prop:taucomp-homomorphisms-collapse-components}.
Isotonicity and multiplicativity are inherited from \(f\), because multiplication and order inside \(\mathbf M_A\) and \(\mathbf N_B\) are the restrictions of the ambient ones.
By Theorem~\ref{thm:component-local-unit-aligned-monoid}, the intrinsic local-unit maps of the components are the ambient local-unit maps restricted to the components.
Hence, for \(x\in M_A\),
\[
\tau_{\mathbf N_B}(f_A(x))
=\tau_{\mathbf N}(f(x))
=f(\tau_{\mathbf M}(x))
=f_A(\tau_{\mathbf M_A}(x)).
\]
Finally, \(e_A\in A\), so \(f(e_A)\in B\) by the definition of \(B\).
Since \(e_B\) is the identity of \(\mathbf N_B\), the final assertion is immediate.
\end{proof}

\begin{remark}\label{rem:taucomp-vs-block-unit-preservation}
In contrast with Definition~\ref{def:cat-lua-blk-finite}, a \(\tau\)-compatible homomorphism need not satisfy \(f(e_A)=e_{\sigma_f(A)}\) for every canonical block \(A\).
Corollary~\ref{cor:taucomp-component-restrictions} shows only that \(f(e_A)\in \sigma_f(A)=\PosId(\mathbf N_{\sigma_f(A)})\), so the restriction \(f_A\colon M_A\to N_{\sigma_f(A)}\) is multiplicative and commutes with the intrinsic local-unit maps, but need not be unital as a map between component monoids.

This failure is genuine.
In Example~\ref{ex:block-preserving-not-unit-preserving}, the map \(f\colon\mathbf S\to\mathbf T\) is \(\tau\)-compatible, but for \(A=\{2\}\in\Cm(\mathbf S)\) and \(B=\{1,2,4\}\in\Cm(\mathbf T)\) one has \(f(e_A)=f(2)=2\neq 1=e_B\).
Thus \(\tau\)-compatibility preserves the global unit and sends each source block unit to a positive idempotent in the appropriate target component, but it does not force that idempotent to be the target block unit.
\end{remark}

\begin{example}[A strict block morphism need not be \(\tau\)-compatible]
\label{ex:strict-block-not-tau-compatible}
Let \(\mathbf T\) be the four-element monoid from Example~\ref{ex:block-preserving-not-unit-preserving}.
Thus
\[
\Cm(\mathbf T)=\bigl\{\{1,2,4\}\bigr\},
\]
so \(\mathbf T\) has only one canonical block, whose block unit is \(1\).

Define \(g\colon \mathbf T\to\mathbf T\) by
\[
g(1)=1,\qquad g(2)=1,\qquad g(3)=2,\qquad g(4)=2.
\]
The map \(g\) is isotone and unital.
It is also multiplicative.
Indeed, the claim is immediate if one factor is \(1\).
If both factors lie in \(\{2,3,4\}\), then their product is \(2\) only in the case \((2,2)\), and is \(4\) otherwise; hence the displayed definition gives \(g(xy)=g(x)g(y)\) in all cases.

Since \(\mathbf T\) has only one canonical block and \(g(1)=1\), the map \(g\) is a strict block morphism \(\mathbf T\to\mathbf T\).
However, it is not \(\tau\)-compatible, because
\[
\tau_{\mathbf T}(g(3))=\tau_{\mathbf T}(2)=2,
\qquad
g(\tau_{\mathbf T}(3))=g(1)=1.
\]
Thus \(\tau_{\mathbf T}g\ne g\tau_{\mathbf T}\).
\end{example}

\begin{definition}\label{def:lua-taucomp-category}
Let \(\mathbf{LUA}^{\tau\mathrm{c}}_{\mathrm{fin}}\) be the category whose objects are finite local-unit-aligned totally ordered monoids and whose morphisms are \(\tau\)-compatible homomorphisms.
Composition is ordinary composition of maps.
\end{definition}

\begin{proposition}\label{prop:lua-taucomp-category}
The objects and morphisms of Definition~\ref{def:lua-taucomp-category} form a category.
\end{proposition}

\begin{proof}
The identity map is \(\tau\)-compatible.
Lemma~\ref{lem:taucomp-layer-preservation} shows that \(\tau\)-compatible homomorphisms are closed under composition.
The category axioms are therefore inherited from ordinary composition of maps.
\end{proof}

\begin{definition}[\(\tau\)-compatible collapse morphism of rigid direct systems]
\label{def:taucomp-collapse-rigds-morphism}
Let
\[
\mathbf D=\bigl((\mathbf M_i)_{i\in I},(\rho^{\mathbf D}_{i,k})_{i\le k}\bigr),
\qquad
\mathbf E=\bigl((\mathbf N_j)_{j\in J},(\rho^{\mathbf E}_{j,\ell})_{j\le \ell}\bigr)
\]
be rigid direct systems.
A \emph{\(\tau\)-compatible collapse morphism} \(\Phi\colon\mathbf D\to\mathbf E \) is a pair \(\Phi=(\sigma,(\Phi_i)_{i\in I}) \) where \(\sigma\colon I\to J\) is isotone and \(\Phi_i\colon M_i\to N_{\sigma(i)}\) is a map for each \(i\in I\), such that the piecewise map
\[
\widehat\Phi\colon\mathfrak M(\mathbf D)\to\mathfrak M(\mathbf E),
\qquad
\widehat\Phi\!\upharpoonright_{M_i}:=\Phi_i,
\]
is a \(\tau\)-compatible homomorphism.

The identity morphism is \((\mathrm{id}_I,(\mathrm{id}_{\mathbf M_i})_{i\in I}), \) and the composite of
\[
\Phi=(\sigma,(\Phi_i)_{i\in I})\colon\mathbf D\to\mathbf E,
\qquad
\Psi=(\nu,(\Psi_j)_{j\in J})\colon\mathbf E\to\mathbf F
\]
is
\[
\Psi\circ\Phi
:=
\bigl(\nu\circ\sigma,(\Psi_{\sigma(i)}\circ\Phi_i)_{i\in I}\bigr).
\]
\end{definition}

The definition of a \(\tau\)-compatible collapse morphism is intentionally given through its realization.
Thus the morphism class used below is not defined by imposing a strict componentwise block-preservation condition.
Rather, it records exactly those componentwise data whose realized map is a \(\tau\)-compatible homomorphism of the reconstructed monoids.
The next proposition shows that this apparently external requirement has an intrinsic direct-system formulation.

\begin{proposition}[Internal form of a \(\tau\)-compatible collapse morphism]
\label{prop:taucomp-collapse-internal-consequences}
Let
\[
\mathbf D=\bigl((\mathbf M_i)_{i\in I},(\rho^{\mathbf D}_{i,k})_{i\le k}\bigr),
\qquad
\mathbf E=\bigl((\mathbf N_j)_{j\in J},(\rho^{\mathbf E}_{j,\ell})_{j\le \ell}\bigr)
\]
be rigid direct systems.
Let \(\sigma\colon I\to J\) be isotone, and for each \(i\in I\) let \(\Phi_i\colon M_i\to N_{\sigma(i)} \) be a map.
Put
\[
\widehat\Phi\colon\mathfrak M(\mathbf D)\to\mathfrak M(\mathbf E),
\qquad
\widehat\Phi\!\upharpoonright_{M_i}:=\Phi_i .
\]
Then \(\Phi=(\sigma,(\Phi_i)_{i\in I}) \) is a \(\tau\)-compatible collapse morphism if and only if the following conditions hold.

\begin{enumerate}
\taggeditem{\textup{(U)}}{item:tccm-unit}
\emph{Global unit.} If \(i_0=\min I\) and \(j_0=\min J\), then \(\sigma(i_0)=j_0\) and \(\Phi_{i_0}(e_{i_0})=e_{j_0}\).

\taggeditem{\textup{(C)}}{item:tccm-component}
\emph{Componentwise \(\tau\)-compatibility.} For each \(i\in I\), the map \(\Phi_i\colon \mathbf M_i\to \mathbf N_{\sigma(i)} \) is isotone and multiplicative, and it commutes with the intrinsic local-unit maps:
\[
\tau_{\mathbf N_{\sigma(i)}}(\Phi_i(x))
=
\Phi_i(\tau_{\mathbf M_i}(x))
\qquad(x\in M_i).
\]
The map \(\Phi_i\) is not required to be unital.

\taggeditem{\textup{(F)}}{item:tccm-forward}
\emph{Forward cross-level order.} Whenever \(i\le k\) and \(x\in M_i\),
\[
\rho^{\mathbf E}_{\sigma(i),\sigma(k)}(\Phi_i(x))
\le_{\sigma(k)}
\Phi_k(\rho^{\mathbf D}_{i,k}(x)).
\]

\taggeditem{\textup{(R)}}{item:tccm-reverse}
\emph{Reversed cross-level order.} Whenever \(i<k, x\in M_k, y\in M_i, x<_k \rho^{\mathbf D}_{i,k}(y), \) one has
\[
\begin{aligned}
\Phi_k(x)&\le_{\sigma(i)}\Phi_i(y)
&&\text{if } \sigma(i)=\sigma(k),\\
\Phi_k(x)&<_{\sigma(k)}
\rho^{\mathbf E}_{\sigma(i),\sigma(k)}(\Phi_i(y))
&&\text{if } \sigma(i)<\sigma(k).
\end{aligned}
\]

\taggeditem{\textup{(A)}}{item:tccm-collapse-absorption}
\emph{Collapse absorption.} Whenever two source components are collapsed into the same target component, cross-level products are absorbed in that target component: if \(i<k\) and \(\sigma(i)=\sigma(k)\), then
\[
\Phi_i(x)\Phi_k(y)=\Phi_k(y)=\Phi_k(y)\Phi_i(x)
\qquad (x\in M_i,\ y\in M_k),
\]
where the products are computed in \(\mathbf N_{\sigma(i)}=\mathbf N_{\sigma(k)}\).
\end{enumerate}
\end{proposition}

\begin{proof}
Assume first that \(\Phi\) is a \(\tau\)-compatible collapse morphism.
Then \(\widehat\Phi\) is an isotone unital homomorphism commuting with the ambient local-unit maps.
Since the global units of \(\mathfrak M(\mathbf D)\) and \(\mathfrak M(\mathbf E)\) are \(e_{i_0}\) and \(e_{j_0}\), unitality gives \(\Phi_{i_0}(e_{i_0})=e_{j_0}\).
As this element lies both in \(N_{\sigma(i_0)}\) and in \(N_{j_0}\), disjointness of the construction components gives \(\sigma(i_0)=j_0\).
This proves condition~\ref{item:tccm-unit}.

Restricting the isotone multiplicative map \(\widehat\Phi\) to a single component gives isotonicity and multiplicativity of each \(\Phi_i\).
By Theorem~\ref{thm:representation}\ref{item:representation-direct-system-recovery}, in a rigid direct-system reconstruction the construction components are the canonical components.
Hence the intrinsic local-unit maps of \(\mathbf M_i\) and \(\mathbf N_{\sigma(i)}\) are the restrictions of the ambient local-unit maps.
Therefore, for \(x\in M_i\),
\[
\tau_{\mathbf N_{\sigma(i)}}(\Phi_i(x))
=
\tau_{\mathfrak M(\mathbf E)}(\widehat\Phi(x))
=
\widehat\Phi(\tau_{\mathfrak M(\mathbf D)}(x))
=
\Phi_i(\tau_{\mathbf M_i}(x)).
\]
This proves \ref{item:tccm-component}.

For \ref{item:tccm-forward}, let \(i\le k\) and \(x\in M_i\).
In \(\mathfrak M(\mathbf D)\), the directed lexicographic order gives \(x\le \rho^{\mathbf D}_{i,k}(x). \)
Applying the isotone map \(\widehat\Phi\), and then using \(\sigma(i)\le\sigma(k)\), gives
\[
\rho^{\mathbf E}_{\sigma(i),\sigma(k)}(\Phi_i(x))
\le_{\sigma(k)}
\Phi_k(\rho^{\mathbf D}_{i,k}(x)).
\]

For \ref{item:tccm-reverse}, suppose that \(i<k\), \(x\in M_k\), \(y\in M_i\), and \(x<_k\rho^{\mathbf D}_{i,k}(y)\).
By the working form of the directed lexicographic order, this is exactly the assertion that \(x\le y\) ambiently in \(\mathfrak M(\mathbf D)\).
Isotonicity of \(\widehat\Phi\) gives \(\Phi_k(x)\le \Phi_i(y)\) ambiently in \(\mathfrak M(\mathbf E)\).
If \(\sigma(i)=\sigma(k)\), this is the internal inequality \(\Phi_k(x)\le_{\sigma(i)}\Phi_i(y)\).
If \(\sigma(i)<\sigma(k)\), the same working form of the directed lexicographic order gives \(\Phi_k(x)<_{\sigma(k)} \rho^{\mathbf E}_{\sigma(i),\sigma(k)}(\Phi_i(y)). \)
Finally, since \(\mathbf D\) is rigid, all proper transition maps in \(\mathbf D\) are unit-constant.
Thus, for \(i<k\), \(xy=y=yx\) inside \(\mathfrak M(\mathbf D)\), for all \(x\in M_i\) and \(y\in M_k\).
Applying multiplicativity of \(\widehat\Phi\), and assuming \(\sigma(i)=\sigma(k)\), gives \(\Phi_i(x)\Phi_k(y)=\Phi_k(y)=\Phi_k(y)\Phi_i(x) \) inside the common target component.
This proves \ref{item:tccm-collapse-absorption}.

Conversely, assume that conditions \ref{item:tccm-unit}--\ref{item:tccm-collapse-absorption} hold.
We prove that \(\widehat\Phi\) is a \(\tau\)-compatible homomorphism.

Condition~\ref{item:tccm-unit} says precisely that \(\widehat\Phi\) preserves the global unit.
We next prove isotonicity.
Let \(u\in M_i\) and \(v\in M_k\), and suppose \(u\le v\) in \(\mathfrak M(\mathbf D)\).
If \(i=k\), isotonicity follows from \ref{item:tccm-component}.
If \(i<k\), then \(\rho^{\mathbf D}_{i,k}(u)\le_k v\).
By \ref{item:tccm-forward} and isotonicity of \(\Phi_k\),
\[
\rho^{\mathbf E}_{\sigma(i),\sigma(k)}(\Phi_i(u))
\le_{\sigma(k)}
\Phi_k(\rho^{\mathbf D}_{i,k}(u))
\le_{\sigma(k)}
\Phi_k(v),
\]
which gives \(\Phi_i(u)\le \Phi_k(v)\) in the directed lexicographic order of \(\mathfrak M(\mathbf E)\).
If \(k<i\), then \(u<_i\rho^{\mathbf D}_{k,i}(v)\).
Applying \ref{item:tccm-reverse}, with \(k\) now playing the lower-index role, gives exactly the working-form condition for \(\Phi_i(u)\le \Phi_k(v)\) in \(\mathfrak M(\mathbf E)\).
Thus \(\widehat\Phi\) is isotone.

We now prove multiplicativity.
Let \(u\in M_i\) and \(v\in M_k\).
If \(i=k\), multiplicativity follows from \ref{item:tccm-component}.
Suppose \(i<k\).
Since \(\mathbf D\) is rigid, \(uv=v=vu\) in \(\mathfrak M(\mathbf D)\).
If \(\sigma(i)<\sigma(k)\), rigidity of \(\mathbf E\) makes the proper transition map \(\rho^{\mathbf E}_{\sigma(i),\sigma(k)}\) unit-constant, so \(\Phi_i(u)\Phi_k(v)=\Phi_k(v)=\Phi_k(v)\Phi_i(u) \) in \(\mathfrak M(\mathbf E)\).
If \(\sigma(i)=\sigma(k)\), the same conclusion is exactly condition \ref{item:tccm-collapse-absorption}.
Hence
\[
\widehat\Phi(uv)=\widehat\Phi(u)\widehat\Phi(v)
\qquad\text{and}\qquad
\widehat\Phi(vu)=\widehat\Phi(v)\widehat\Phi(u).
\]
The case \(k<i\) is symmetric.
Therefore \(\widehat\Phi\) is multiplicative.

It remains to check compatibility with the ambient local-unit maps.
If \(x\in M_i\), then, again using Theorem~\ref{thm:representation}\ref{item:representation-direct-system-recovery} to identify construction components with canonical components in the two reconstructions,
\[
\tau_{\mathfrak M(\mathbf E)}(\widehat\Phi(x))
=
\tau_{\mathbf N_{\sigma(i)}}(\Phi_i(x))
=
\Phi_i(\tau_{\mathbf M_i}(x))
=
\widehat\Phi(\tau_{\mathfrak M(\mathbf D)}(x)),
\]
where the middle equality is condition \ref{item:tccm-component}.
Thus \(\widehat\Phi\) is a \(\tau\)-compatible homomorphism, and consequently \(\Phi\) is a \(\tau\)-compatible collapse morphism.
\end{proof}

The order of presentation is intentional.
One could instead take the conditions in Proposition~\ref{prop:taucomp-collapse-internal-consequences} as the primitive definition of a collapse morphism.
That formulation is useful for checking examples, but it would obscure the reason for this particular list of conditions.
The forward and reverse cross-level inequalities are exactly the direct-system translation of isotonicity of the realized map, while the collapse-absorption condition is exactly the extra multiplicativity requirement that appears when distinct source components are sent into the same target component.
Thus the realization-based definition puts the invariant first--namely, that the assembled piecewise map is a \(\tau\)-compatible homomorphism of the reconstructed monoids--and the proposition then identifies the corresponding internal direct-system criterion.
It also makes closure under composition immediate; a direct proof from the internal conditions alone would require a separate verification that the order and absorption clauses survive all possible compositions of index collapses.

\begin{proposition}\label{prop:rigds-taucomp-category}
Rigid direct systems and \(\tau\)-compatible collapse morphisms form a category.
We denote it by \(\mathbf{RigDS}^{\tau\mathrm{c}}_{\mathrm{fin}}. \)
\end{proposition}

\begin{proof}
The identity morphism realizes the identity \(\tau\)-compatible homomorphism on \(\mathfrak M(\mathbf D)\).
If \(\Phi\colon\mathbf D\to\mathbf E\) and \(\Psi\colon\mathbf E\to\mathbf F\) are \(\tau\)-compatible collapse morphisms, then the piecewise realization of the declared composite is \(\widehat{\Psi\circ\Phi}=\widehat\Psi\circ\widehat\Phi.\)
The right-hand side is a composite of \(\tau\)-compatible homomorphisms, hence is \(\tau\)-compatible by Lemma~\ref{lem:taucomp-layer-preservation}.
Therefore the declared composite is again a \(\tau\)-compatible collapse morphism.
Associativity and the identity laws follow from the componentwise formula for composition.
\end{proof}

\begin{proposition}[Ambient \(\tau\)-compatible maps and \(\tau\)-compatible collapse morphisms]
\label{prop:ambient-taucomp-maps-vs-taucomp-collapse}
Let \(\mathbf D\) and \(\mathbf E\) be rigid direct systems.
The assignment \(\Phi\longmapsto\widehat\Phi \) is a bijection from \(\tau\)-compatible collapse morphisms \(\mathbf D\to\mathbf E\) onto \(\tau\)-compatible homomorphisms \(\mathfrak M(\mathbf D)\to\mathfrak M(\mathbf E). \)
More explicitly, if \(h\colon\mathfrak M(\mathbf D)\to\mathfrak M(\mathbf E) \) is \(\tau\)-compatible, then there is a unique isotone map \(\sigma_h\colon I\to J\) such that \(h(M_i)\subseteq N_{\sigma_h(i)}\) for every \(i\in I\), and the corresponding \(\tau\)-compatible collapse morphism is \((\sigma_h,(h\!\upharpoonright_{M_i})_{i\in I}). \)
\end{proposition}

\begin{proof}
If \(\Phi\) is a \(\tau\)-compatible collapse morphism, then \(\widehat\Phi\) is a \(\tau\)-compatible homomorphism by definition.
Conversely, let \(h\colon\mathfrak M(\mathbf D)\to\mathfrak M(\mathbf E)\) be \(\tau\)-compatible.
For \(i\in I\) and \(j\in J\), put
\[
P_i:=\PosId(\mathfrak M(\mathbf D))\cap M_i,
\qquad
Q_j:=\PosId(\mathfrak M(\mathbf E))\cap N_j.
\]
By Theorem~\ref{thm:representation}\ref{item:representation-direct-system-recovery},
\[
\Cm(\mathfrak M(\mathbf D))=\{P_i:i\in I\},
\qquad
\Cm(\mathfrak M(\mathbf E))=\{Q_j:j\in J\},
\]
and \(\layer{P_i}=M_i\) and \(\layer{Q_j}=N_j\).
Proposition~\ref{prop:taucomp-homomorphisms-collapse-components} therefore gives, for every \(i\in I\), a unique block \(Q_{\sigma_h(i)}\in \Cm(\mathfrak M(\mathbf E))\) such that \(h(M_i)=h(\layer{P_i})\subseteq \layer{Q_{\sigma_h(i)}}=N_{\sigma_h(i)}. \)
The same proposition also gives isotonicity of the induced block map \(P_i\mapsto Q_{\sigma_h(i)}\).
Since \(i\mapsto P_i\) and \(j\mapsto Q_j\) are order isomorphisms, \(\sigma_h\colon I\to J\) is isotone.
Hence \((\sigma_h,(h\!\upharpoonright_{M_i})_{i\in I})\) is a \(\tau\)-compatible collapse morphism whose realization is \(h\).

The two constructions are inverse to one another.
Indeed, the source layers \(M_i=\layer{P_i}\) are pairwise disjoint, so the component maps are determined by the realized ambient map.
The index value is also determined, because each nonempty set \(h(M_i)\) is contained in exactly one canonical layer \(N_j=\layer{Q_j}\) of \(\mathfrak M(\mathbf E)\).
\end{proof}

\begin{theorem}[Categorical equivalence for \(\tau\)-compatible homomorphisms]
\label{thm:taucomp-categorical-equivalence}
The assignments
\[
\mathfrak D^{\tau\mathrm{c}}\colon
\mathbf{LUA}^{\tau\mathrm{c}}_{\mathrm{fin}}
\longrightarrow
\mathbf{RigDS}^{\tau\mathrm{c}}_{\mathrm{fin}},
\qquad
\mathfrak M^{\tau\mathrm{c}}\colon
\mathbf{RigDS}^{\tau\mathrm{c}}_{\mathrm{fin}}
\longrightarrow
\mathbf{LUA}^{\tau\mathrm{c}}_{\mathrm{fin}}
\]
form quasi-inverse equivalences of categories.
\end{theorem}

\begin{proof}
On objects, the two assignments are the canonical object assignments from Theorem~\ref{thm:representation}: \(\mathfrak D^{\tau\mathrm{c}}\) sends a finite local-unit-aligned totally ordered monoid to its canonical rigid direct system, and \(\mathfrak M^{\tau\mathrm{c}}\) reconstructs the ambient monoid of a rigid direct system.

On morphisms, let \(f\colon\mathbf M\to\mathbf N\) be \(\tau\)-compatible.
By Proposition~\ref{prop:taucomp-homomorphisms-collapse-components}, for every \(A\in\Cm(\mathbf M)\) there is a unique block \(\sigma_f(A)\in\Cm(\mathbf N)\) such that \(f[M_A]\subseteq N_{\sigma_f(A)},\) and \(\sigma_f\) is isotone.
Define
\[
\mathfrak D^{\tau\mathrm{c}}(f)
:=
\bigl(\sigma_f,(f\!\upharpoonright_{M_A})_{A\in\Cm(\mathbf M)}\bigr).
\]
Let \(\widehat f\) be its piecewise realization.
With the canonical comparison isomorphisms from Theorem~\ref{thm:representation}\ref{item:representation-canonical-recovery}, one has the typed equality
\[
\eta_{\mathbf N}\circ \widehat f=f\circ \eta_{\mathbf M}.
\]
Equivalently, \(\widehat f=\eta_{\mathbf N}^{-1}\circ f\circ\eta_{\mathbf M}\).
The maps \(\eta_{\mathbf M}\) and \(\eta_{\mathbf N}\) are \(\tau\)-compatible isomorphisms, so \(\widehat f\) is \(\tau\)-compatible.
Thus \(\mathfrak D^{\tau\mathrm{c}}(f)\) is a \(\tau\)-compatible collapse morphism.

Conversely, if \(\Phi=(\sigma,(\Phi_i)_{i\in I})\colon\mathbf D\to\mathbf E\) is a \(\tau\)-compatible collapse morphism, define \(\mathfrak M^{\tau\mathrm{c}}(\Phi):=\widehat\Phi.\)
By Definition~\ref{def:taucomp-collapse-rigds-morphism}, this is a \(\tau\)-compatible homomorphism \(\mathfrak M(\mathbf D)\to\mathfrak M(\mathbf E).\)

Functoriality of \(\mathfrak M^{\tau\mathrm{c}}\) follows from \(\widehat{\Psi\circ\Phi}=\widehat\Psi\circ\widehat\Phi\).
Functoriality of \(\mathfrak D^{\tau\mathrm{c}}\) follows from uniqueness in Proposition~\ref{prop:ambient-taucomp-maps-vs-taucomp-collapse}: for \(\tau\)-compatible \(f\colon\mathbf M\to\mathbf N\) and \(g\colon\mathbf N\to\mathbf P\), both \(\mathfrak D^{\tau\mathrm{c}}(g\circ f)\) and \(\mathfrak D^{\tau\mathrm{c}}(g)\circ\mathfrak D^{\tau\mathrm{c}}(f)\) realize the same typed ambient homomorphism \(\eta_{\mathbf P}^{-1}\circ g\circ f\circ\eta_{\mathbf M}\).
The identity case is immediate.

It remains to identify the quasi-inverse isomorphisms.
On the monoid side, for every \(\mathbf M\) the canonical comparison
\[
\eta_{\mathbf M}\colon
\mathfrak M^{\tau\mathrm{c}}(\mathfrak D^{\tau\mathrm{c}}(\mathbf M))
\longrightarrow
\mathbf M
\]
is induced by the identity maps on the canonical components.
In the tagged copy convention, it sends the reconstructed copy of \(x\in M_A\) to the element \(x\) of \(\mathbf M\).
It is therefore an isomorphism of finite local-unit-aligned totally ordered monoids and is \(\tau\)-compatible.

On the direct-system side, let \(\mathbf D=\bigl((\mathbf M_i)_{i\in I},(\rho^{\mathbf D}_{i,k})_{i\le k}\bigr)\) be rigid, and put
\[
P_i:=\PosId(\mathfrak M(\mathbf D))\cap M_i
\qquad(i\in I).
\]
By Theorem~\ref{thm:representation}\ref{item:representation-direct-system-recovery}, the construction-level layers are exactly the canonical layers:
\[
\Cm(\mathfrak M(\mathbf D))=\{P_i:i\in I\},
\qquad
\layer{P_i}=M_i.
\]
Moreover, the map
\[
\kappa_{\mathbf D}\colon I\to \Cm(\mathfrak M(\mathbf D)),
\qquad
\kappa_{\mathbf D}(i):=P_i,
\]
is an order isomorphism, the canonical component indexed by \(P_i\) is \(\mathbf M_i\), and its block unit is \(e_i\).
Let \(\alpha_{\mathbf D}:=\eta_{\mathfrak M(\mathbf D)}^{-1}\).
Then \(\alpha_{\mathbf D}\) yields the \(\tau\)-compatible collapse morphism
\[
\varepsilon_{\mathbf D}
=
\bigl(\kappa_{\mathbf D},(\mathrm{id}_{\mathbf M_i})_{i\in I}\bigr)
\colon
\mathbf D\to
\mathfrak D^{\tau\mathrm{c}}(\mathfrak M^{\tau\mathrm{c}}(\mathbf D)).
\]
Similarly let \(\beta_{\mathbf D}:=\eta_{\mathfrak M(\mathbf D)}\).
For each \(i\in I\), the map \(\beta_{\mathbf D}\) sends the reconstructed canonical component indexed by \(P_i=\kappa_{\mathbf D}(i)\) onto the original component \(M_i\), and sends its block unit to \(e_i\).
Hence Proposition~\ref{prop:ambient-taucomp-maps-vs-taucomp-collapse} yields a unique \(\tau\)-compatible collapse morphism
\[
\delta_{\mathbf D}\colon
\mathfrak D^{\tau\mathrm{c}}(\mathfrak M^{\tau\mathrm{c}}(\mathbf D))
\to
\mathbf D
\]
whose realization is \(\beta_{\mathbf D}\).
Its induced index map is \(\kappa_{\mathbf D}^{-1}\), since the component over \(P_i\) is mapped onto \(M_i\).
The composites \(\delta_{\mathbf D}\circ\varepsilon_{\mathbf D}\) and \(\varepsilon_{\mathbf D}\circ\delta_{\mathbf D}\) realize, respectively, \(\beta_{\mathbf D}\circ\alpha_{\mathbf D}=\mathrm{id}_{\mathfrak M(\mathbf D)}\) and \(\alpha_{\mathbf D}\circ\beta_{\mathbf D} =\mathrm{id}_{\mathfrak M(\mathfrak D(\mathfrak M(\mathbf D)))}\).
By the uniqueness part of Proposition~\ref{prop:ambient-taucomp-maps-vs-taucomp-collapse}, they are the corresponding identity morphisms.
Thus \(\varepsilon_{\mathbf D}\) is an isomorphism in \(\mathbf{RigDS}^{\tau\mathrm{c}}_{\mathrm{fin}}\).

Naturality of \((\eta_{\mathbf M})\) is the typed equality \(\eta_{\mathbf N}\circ \mathfrak M^{\tau\mathrm{c}}(\mathfrak D^{\tau\mathrm{c}}(f)) =f\circ\eta_{\mathbf M}\).
Naturality of \((\varepsilon_{\mathbf D})\) follows from the same uniqueness statement: for a collapse morphism \(\Phi\colon\mathbf D\to\mathbf E\), both composites in the naturality square realize the typed ambient map \(\eta_{\mathfrak M(\mathbf E)}^{-1}\circ \mathfrak M^{\tau\mathrm{c}}(\Phi)\).
Hence the two assignments are quasi-inverse equivalences of categories.
\end{proof}

\begin{remark}\label{rem:role-of-rigidity-taucomp-equivalence}
Rigidity enters the \(\tau\)-compatible equivalence in two essential ways.
First, it ensures that \(\tau\)-compatible homomorphisms induce well-defined maps on canonical components.
The reason is that rigidity makes every canonical component \(\tau\)-multiplication-cohesive, and this prevents a \(\tau\)-compatible homomorphism from splitting one source component among several target components, see Proposition~\ref{prop:taucomp-homomorphisms-collapse-components}.
Second, rigidity supplies the unit-constancy of proper transition maps, which makes a non-injective index map \(\sigma\) behave exactly like a component collapse.
If \(i<k\) and \(\sigma(i)=\sigma(k)\), then the image of the higher component absorbs the image of the lower component inside the common target component, as in Proposition~\ref{prop:taucomp-collapse-internal-consequences}.
Thus the direct-system side records idempotent collapse not by identifying underlying elements, but by allowing several source components to land in the same target component.
\end{remark}

\begin{remark}[Residuated origins and \texorpdfstring{\(\tau\)}{tau}-compatible morphisms]\label{rem:not-res}
The basic local-unit stratification originates in the residuated setting of \cite{Jenei2022GroupRepr}, where the decomposition is organized by the exact local-unit layers \(L_u=\{x:\tau(x)=u\}\).
In \cite{JeneiLUARepresentation}, this layer structure was developed further in the finite local-unit-aligned ordered-monoid setting by introducing \(\tau\)-saturated products, \(\tau\)-multiplication-coherent partitions, and the canonical rigid direct-system representation.

As explained in Remark~8.13 of \cite{JeneiLUARepresentation}, this refined decomposition is not residuation-theoretic in nature: even when the ambient monoid is the monoidal reduct of a residuated chain, its component monoids need not themselves be residuated.
Accordingly, the categorical theory developed below is monoid-theoretic rather than residuation-theoretic: its reconstruction arguments and morphism classes use only the ordered-monoid structure, the local-unit map, and the canonical rigid direct system.

There is nevertheless a natural connection at the level of morphisms.
Suppose that the local-unit-aligned monoid is the monoidal reduct of a residuated chain, with residuals denoted by \(\backslash\) and \(/\).
In a residuated monoid, \(x/x\) is the greatest right local unit of \(x\), while \(x\backslash x\) is the greatest left local unit of \(x\).
In the local-unit-aligned case these two elements coincide, and their common value is \(\tau(x)\).
Consequently, if \(h\colon \mathbf A\to \mathbf B\) is a homomorphism of the corresponding residuated structures, so that \(h\) preserves multiplication and both residuals, then for every \(x\in A\) one has \(h(\tau_{\mathbf A}(x))=h(x/x)=h(x)/h(x)=\tau_{\mathbf B}(h(x))\).
Equivalently, \(h\tau_{\mathbf A}=\tau_{\mathbf B}h\).
Thus residuated homomorphisms are automatically \(\tau\)-compatible homomorphisms in the sense used in the second categorical comparison below.
\end{remark}

\section*{Conclusion}

We have shown that the canonical rigid direct-system representation of finite local-unit-aligned totally ordered monoids is functorial in two natural senses.
The first functorial form uses strict block morphisms.
In that strict setting, morphisms of monoids correspond exactly to directed-order-compatible morphisms of the associated rigid direct systems, and the decomposition and reconstruction assignments form quasi-inverse equivalences of categories.

The second form uses the intrinsic local-unit condition \(\tau_{\mathbf N}f=f\tau_{\mathbf M}\).
The two morphism classes control different structure.
A \(\tau\)-compatible homomorphism is formulated entirely in terms of the local-unit map and may fail to preserve block units; conversely, a strict block morphism is not defined by requiring commutation with \(\tau\).
Example~\ref{ex:block-preserving-not-unit-preserving} shows that a \(\tau\)-compatible homomorphism need not be a strict block morphism, while Example~\ref{ex:strict-block-not-tau-compatible} shows that a strict block morphism need not be \(\tau\)-compatible.
Thus the two categorical comparisons are complementary rather than nested.

In the \(\tau\)-compatible setting, several source components may collapse into one target component.
On the direct-system side, this collapse is encoded by non-injective isotone maps between index chains and by component maps satisfying the corresponding compatibility, unit, and absorption conditions.
Rigidity is essential in this second equivalence, because the unit-constant transition maps are exactly what make collapse compatible with multiplication and order.

Thus the representation theorem has a categorical form for the two natural morphism classes considered here: the canonical decomposition is not only an object-level reconstruction device, but also a functorial invariant for both strict block morphisms and intrinsic \(\tau\)-compatible homomorphisms.
In this sense the rigid direct-system representation provides a categorical classification of finite local-unit-aligned totally ordered monoids, with morphisms controlled by the same local-unit structure that governs the object-level reconstruction.

\backmatter

\bmhead{Acknowledgements}

The author gratefully acknowledges support from the Ministry of Culture and Innovation of Hungary, through the National Research, Development and Innovation Fund, grant no.~K138596.
\bigskip

\section*{Declarations}

\noindent\textbf{Competing interests.}
The author declares that he has no competing interests.

\noindent\textbf{Data availability.}
No external datasets were used.
The finite example in Example~\ref{ex:block-preserving-not-unit-preserving} can be checked directly from the displayed multiplication table and the defining conditions.


\begin{thebibliography}{99}

\bibitem{Bonzio2018}
Bonzio, S.:
Dualities for P\l{}onka sums.
\emph{Logica Universalis} \textbf{12} (2018), 327--339.
\newline\url{https://doi.org/10.1007/s11787-018-0209-4}

\bibitem{Bowman1975}
Bowman, T.T.:
Construction functors for topological semigroups.
\emph{Pacific Journal of Mathematics} \textbf{60} (1975), no.~2, 27--36.

\bibitem{Clifford1941}
Clifford, A.H.:
Semigroups admitting relative inverses.
\emph{Annals of Mathematics} \textbf{42} (1941), 1037--1049.
\newline\url{https://doi.org/10.2307/1968781}

\bibitem{Clifford1954}
Clifford, A.H.:
Naturally totally ordered commutative semigroups.
\emph{American Journal of Mathematics} \textbf{76} (1954), 631--646.
\newline\url{https://doi.org/10.2307/2372706}

\bibitem{Clifford1958}
Clifford, A.H.:
Totally ordered commutative semigroups.
\emph{Bulletin of the American Mathematical Society} \textbf{64} (1958), 305--316.
\newline\url{https://doi.org/10.1090/S0002-9904-1958-10221-9}

\bibitem{CliffordPreston1961}
Clifford, A.H., Preston, G.B.:
\emph{The algebraic theory of semigroups, Vol.~I}.
American Mathematical Society, Providence (1961)

\bibitem{DeWolfPronk2018}
DeWolf, D., Pronk, D.:
The Ehresmann--Schein--Nambooripad theorem for inverse categories.
\emph{Theory and Applications of Categories} \textbf{33} (2018), no.~27,
813--831.

\bibitem{FuscoPaoli2025}
Fusco, L., Paoli, F.:
Enriched P\l{}onka sums.
\emph{Algebra Universalis} \textbf{87} (2025), no.~1.
\newline\url{https://doi.org/10.1007/s00012-025-00910-x}

\bibitem{Hollings2012}
Hollings, C.:
The Ehresmann--Schein--Nambooripad theorem and its successors.
\emph{European Journal of Pure and Applied Mathematics} \textbf{5} (2012),
no.~4, 414--450.

\bibitem{JanelidzeLaanMarki2008}
Janelidze, G., Laan, V., M\'arki, L.:
Limit preservation properties of the greatest semilattice image functor.
\emph{International Journal of Algebra and Computation} \textbf{18} (2008),
853--867.
\newline\url{https://doi.org/10.1142/S0218196708004664}

\bibitem{JeneiLUARepresentation}
Jenei, S.:
A canonical rigid direct-system representation of finite local-unit-aligned
totally ordered monoids.
arXiv:2607.23801v1, 2026.
\newline\url{https://doi.org/10.48550/arXiv.2607.23801}

\bibitem{Jenei2022GroupRepr}
Jenei, S.:
Group representation for even and odd involutive commutative residuated chains.
\emph{Studia Logica}
\textbf{110} (2022), 881--922.
First circulated as arXiv:1910.01404 (2019).
\newline\url{https://doi.org/10.1007/s11225-021-09981-y}
\newline\url{https://doi.org/10.48550/arXiv.1910.01404}

\bibitem{Manuell2022}
Manuell, G.:
Monoid extensions and the Grothendieck construction.
\emph{Semigroup Forum} \textbf{105} (2022), 488--507.
\newline\url{https://doi.org/10.1007/s00233-022-10294-2}

\bibitem{Plonka1967}
P\l{}onka, J.:
On a method of construction of abstract algebras.
\emph{Fundamenta Mathematicae} \textbf{61} (1967), 183--189.
\newline\url{https://doi.org/10.4064/fm-61-2-183-189}

\bibitem{Plonka1968}
P\l{}onka, J.:
Some remarks on sums of direct systems of algebras.
\emph{Fundamenta Mathematicae} \textbf{62} (1968), no.~3, 301--308.

\bibitem{RomanowskaSmith1991}
Romanowska, A.B., Smith, J.D.H.:
On the structure of semilattice sums.
\emph{Czechoslovak Mathematical Journal} \textbf{41} (1991), no.~1, 24--43.

\bibitem{RomanowskaSmith1997}
Romanowska, A.B., Smith, J.D.H.:
Duality for semilattice representations.
\emph{Journal of Pure and Applied Algebra} \textbf{115} (1997), 289--308.
\newline\url{https://doi.org/10.1016/S0022-4049(96)00026-6}

\bibitem{Zawadowski2015}
Zawadowski, M.:
Generalized P\l{}onka sums and products.
\emph{Applied Categorical Structures} \textbf{23} (2015), 63--86.
\newline\url{https://doi.org/10.1007/s10485-013-9364-1}

\end{thebibliography}
\end{document}